\documentclass[oneside,english]{amsart}
\usepackage[T1]{fontenc}
\usepackage{color}
\usepackage{amsthm}
\usepackage{amstext}
\usepackage{graphicx}
\usepackage{amssymb}
\usepackage{esint}
\usepackage[all]{xy}

\makeatletter
\numberwithin{equation}{section} 
\numberwithin{figure}{section} 
\theoremstyle{plain}
\newtheorem{thm}{Theorem}[section]
  \theoremstyle{plain}
  \newtheorem{lem}[thm]{Lemma}
  \theoremstyle{plain}
  \newtheorem{prop}[thm]{Proposition}

\usepackage{amsmath}
\usepackage{amsfonts}
\usepackage{amsthm}
\usepackage{amssymb}
\usepackage{amscd}
\usepackage[all]{xy}
\usepackage{enumerate}
\usepackage{mathrsfs}
\usepackage{graphicx, bm, color, ulem,tabls}

\def\ft#1{{\mathsf #1}}
\def\chow{{\mathscr X}}
\def\hcoY{{\mathscr Y}}
\def\Hes{{\mathscr H}}

\def\Zpq{{\mathscr Z}}
\def\Cs{{\mathscr C}}

\def\hchow{{\mathaccent20\chow}}

\def\Lrho{\text{\large{\mbox{$\rho\hskip-2pt$}}}}

\font\eu=eusm10 at 10pt

\def\eF{\text{\eu F}}
\def\eG{\text{\eu G}}
\def\eQ{\text{\eu W}}
\def\eS{\text{\eu U}}

\def\eW{\text{\eu W}}

\def\simarrow{\smash{\mathop{\rightarrow}\limits^{\sim}}} 

\newcommand{\vcorr}[3][1]{%
  \begingroup
    \tabcolsep=.5\tabcolsep
    \sbox0{%
      \begin{tabular}[b]{@{}l}%
        #3%
         \tabularnewline
      \end{tabular}%
    }%
    \settoheight{\dimen0 }{%
      \rotatebox{#2}{%
        \copy0 %
        \kern-\tabcolsep
      }%
    }%
    \rule{0pt}{#1\dimen0}%
    \setlength{\wd0 }{1em}%
    \setlength{\ht0 }{1em}%
    \rotatebox{#2}{\usebox{0}}%
  \endgroup
}

\newenvironment{fcaption}{\begin{list}{}{
\setlength{\leftmargin}{35pt}
\setlength{\rightmargin}{35pt}
\setlength{\labelsep}{5pt}
}}{\end{list}}

\newenvironment{myitem}{\begin{list}{}{
\setlength{\leftmargin}{3pt}
\setlength{\labelsep}{5pt}
\setlength{\itemindent}{0pt}
}}{\end{list}}

\newenvironment{myitem2}{\begin{list}{}{
\setlength{\leftmargin}{0.5cm}
\setlength{\itemindent}{-0.3cm}
\setlength{\itemsep}{0cm}
}}{\end{list}}

\newenvironment{myitem3}{\begin{list}{}{
\setlength{\leftmargin}{0.7cm}
\setlength{\itemindent}{-0.6cm}
\setlength{\itemsep}{0cm}
}}{\end{list}}

\makeatother

\usepackage{babel}

\begin{document}
\global\long\def\sA{\mathcal{A}}
 \global\long\def\sB{\mathcal{B}}
 \global\long\def\sC{\mathcal{C}}
 \global\long\def\sD{\mathcal{D}}
 \global\long\def\sE{\mathcal{E}}
 \global\long\def\sF{\mathcal{F}}
 \global\long\def\sG{\mathcal{G}}
 \global\long\def\sH{\mathcal{H}}
 \global\long\def\sI{\mathcal{I}}
 \global\long\def\sJ{\mathcal{J}}
 \global\long\def\sK{\mathcal{K}}
 \global\long\def\sL{\mathcal{L}}
 \global\long\def\sN{\mathcal{N}}
 \global\long\def\sM{\mathcal{M}}
 \global\long\def\sO{\mathcal{O}}
 \global\long\def\sP{\mathcal{P}}
 \global\long\def\sS{\mathcal{S}}
 \global\long\def\sR{\mathcal{R}}
 \global\long\def\sQ{\mathcal{Q}}
 \global\long\def\sT{\mathcal{T}}
 \global\long\def\sU{\mathcal{U}}
 \global\long\def\sV{\mathcal{V}}
 \global\long\def\sW{\mathcal{W}}
 \global\long\def\sX{\mathcal{X}}
 \global\long\def\sY{\mathcal{Y}}
 \global\long\def\sZ{\mathcal{Z}}
 \global\long\def\tA{{\widetilde{A}}}
 \global\long\def\mA{\mathbb{A}}
 \global\long\def\mC{\mathbb{C}}
 \global\long\def\mF{\mathbb{F}}
 \global\long\def\mG{\mathbb{G}}
 \global\long\def\G{{\bf G}}
 \global\long\def\mN{\mathbb{N}}
 \global\long\def\mP{\mathbb{P}}
 \global\long\def\mQ{\mathbb{Q}}
 \global\long\def\mZ{\mathbb{Z}}
 \global\long\def\mW{\mathbb{W}}
 \global\long\def\Ima{\mathrm{Im}\,}
 \global\long\def\Ker{\mathrm{Ker}\,}
 \global\long\def\Alb{\mathrm{Alb}\,}
 \global\long\def\ap{\mathrm{ap}}
 \global\long\def\Bs{\mathrm{Bs}\,}
 \global\long\def\Chow{\mathrm{Chow}}
 \global\long\def\CP{\mathrm{CP}}
 \global\long\def\Div{\mathrm{Div}\,}
 \global\long\def\divi{\mathrm{div}\,}
 \global\long\def\expdim{\mathrm{expdim}\,}
 \global\long\def\ord{\mathrm{ord}\,}
 \global\long\def\Aut{\mathrm{Aut}\,}
 \global\long\def\Hilb{\mathrm{Hilb}}
 \global\long\def\Hom{\mathrm{Hom}}
 \global\long\def\id{\mathrm{id}}
 \global\long\def\Ext{\mathrm{Ext}}
 \global\long\def\sHom{\mathcal{H}{\!}om\,}
 \global\long\def\Lie{\mathrm{Lie}\,}
 \global\long\def\mult{\mathrm{mult}}
 \global\long\def\opp{\mathrm{opp}}
 \global\long\def\Pic{\mathrm{Pic}\,}
 \global\long\def\Pf{{\bf Pf}}
 \global\long\def\Sec{\mathrm{Sec}}
 \global\long\def\Spec{\mathrm{Spec}\,}
 \global\long\def\Sym{\mathrm{Sym}}
 \global\long\def\sQpec{\mathcal{S}{\!}pec\,}
 \global\long\def\Proj{\mathrm{Proj}\,}
 \global\long\def\Rhom{{\mathbb{R}\mathcal{H}{\!}om}\,}
 \global\long\def\aw{\mathrm{aw}}
 \global\long\def\exc{\mathrm{exc}\,}
 \global\long\def\emb{\mathrm{emb\text{-}dim}}
 \global\long\def\codim{\mathrm{codim}\,}
 \global\long\def\OG{\mathrm{OG}}
 \global\long\def\pr{\mathrm{pr}}
 \global\long\def\Sing{\mathrm{Sing}\,}
 \global\long\def\Supp{\mathrm{Supp}\,}
 \global\long\def\SL{\mathrm{SL}\,}
 \global\long\def\Reg{\mathrm{Reg}\,}
 \global\long\def\rank{\mathrm{rank}\,}
 \global\long\def\VSP{\mathrm{VSP}\,}
 \global\long\def\B{B}
 \global\long\def\Q{Q}
 \global\long\def\rG{\mathrm{G}}
 \global\long\def\rk{\mathrm{rk}\,}
\global\long\def\tS{\ft S}

\hfill 

\begin{center}
\textbf{\Large Mirror Symmetry and Projective Geometry of }
\par\end{center}{\Large \par}

\begin{center}
\textbf{\Large Fourier-Mukai Partners}
\par\end{center}{\Large \par}

$\;$

$\;$

$\;$

\begin{center}
Shinobu Hosono and Hiromichi Takagi 
\par\end{center}

\markboth{}{Mirror Symmetry and Projective Geometry of FM partners}

$\;$

\begin{fcaption} {\small  \item 

Abstract. This is a survey article on mirror symmetry and Fourier-Mukai
partners of Calabi-Yau threefolds with Picard number one based on
recent works \cite[\hskip-3pt 2,3,4]{HoTa1}. For completeness, mirror
symmetry and Fourier-Mukai partners of K3 surfaces are also discussed. 

}\end{fcaption}

\vspace{0.5cm}
 { \tableofcontents{} }

\newpage

\section{\textbf{\textup{Introduction}}}

Derived categories of coherent sheaves on projective varieties are
attracting attentions from many aspects of mathematics for the last
decades. Among them, the derived categories of coherent sheaves on
Calabi-Yau manifolds have been attracting special attentions since
they are conjecturally related to symplectic geometry by the homological
mirror symmetry due to Kontsevich \cite{Ko} and also to the geometric
mirror symmetry due to Strominger-Yau-Zaslow \cite{SYZ}. In this
article, we will survey on the derived categories of Calabi-Yau manifolds
of dimension two and three focusing on the so-called Fourier-Mukai
partners and their mirror symmetry. 

As defined in the text, smooth projective projective varieties $X$
and $Y$ are called Fourier-Mukai partners to each other if their
derived categories of bounded complexes of coherent sheaves are equivalent,
$D^{b}(X)\simeq D^{b}(Y)$. When $X$ and $Y$ are K3 surfaces, the
study of the derived equivalence goes back to the works by Mukai in
'80s \cite{Mukai1} and Orlov in '90s \cite{Orlov}. For completeness,
we start our survey with a brief summary of their results, and also
the mirror symmetry interpretations made in \cite{HLOYauteq}. About
the Fourier-Mukai partners of Calabi-Yau threefolds, little is known
except a general result that two Calabi-Yau threefolds are derived
equivalent if they are birational \cite{Br2}. In \cite{BC}\cite{Ku2},
it has been shown that an interesting example of a pair of Calabi-Yau
threefolds $X$, $Y$ of Picard number one (Grassmannian-Pfaffian
Calabi-Yau threefolds) due to R{\o}dland \cite{Ro} is the case of
non-trivial Fourier-Mukai partners which are not birational. In particular,
it has been recognized in \cite{Ku2,Ku1} that the classical projective
duality between the Grassmannian $\rG(2,7)$ and the Pfaffian variety
$\mathrm{Pf}(4,7)$ in the construction of $X$ and $Y$ plays a prominent
role, and a notion called homological projective duality has been
introduced in \cite{Ku1}. \textcolor{black}{Recently, it has been
found by the present authors \cite[\hskip-3pt2,3,4]{HoTa1} that }the
projective duality of $\rG(2,7)$ and $\mathrm{Pf(4,7)}$\textcolor{red}{{}
}\textcolor{black}{has a natural counterpart in the projective duality
between the secant varieties }of symmetric forms \textcolor{black}{and
these of the} dual forms. In this setting, we naturally came to two
Calabi-Yau threefolds $X$ and $Y$ of Picard numbers one which are
derived equivalent but not birational to each other. Calabi-Yau manifold
$X$ is the so-called three dimensional Reye congruence (whose two
dimensional counterpart has been studied in \cite{Cossec}), and $Y$
is given by a linear section of double quintic symmetroids (see Section
\ref{sec:Birational-Geometry-of-hcoY}). 

In the construction of $Y$ and also in the proof of the derived equivalence
to $X$, birational geometry of the double quintic symmetroids has
been worked out in detail in \cite{HoTa3}. It has been found that
the birational geometry of symmetroids itself contains interesting
projective geometry of quadrics \cite{Tyurin}. 

This article is aimed to be a survey of the works {[}HoTa1,2,3,4{]}
on mirror symmetry and Fourier-Mukai partners of the new Calabi-Yau
manifolds of Picard number one, and also interesting birational geometry
of the double quintic symmetroids which arises in the constructions.
In order to clarify the entire picture of the subjects, we have included
previous works on K3 surfaces and also the R{\o}dland's example.
Since the expository nature of this article, most of the proofs for
the statements are omitted referring to the original papers. \vskip0.3cm

\noindent \textbf{Acknowledgements:}\textcolor{red}{{} }\textcolor{black}{The
first named author would like to thank K. Oguiso, B.H. Lian and S.-T.
Yau for valuable collaborations on Fourier-Mukai partners of K3 surfaces.
}This article is supported in part by Grant-in Aid Scientific Research
(C 18540014, S.H.) and Grant-in Aid for Young Scientists (B 20740005,
H.T.).

\section{\textbf{\textup{\label{sec:Fourier-Mukai-partners-K3}Fourier-Mukai
partners of K3 surfaces}}}

\subsection{Counting formula of Fourier-Mukai partners }

Let $X$ be a K3 surface, i.e., a smooth projective surface with $K_{X}\simeq\sO_{X}$
and $H^{1}(X,\sO_{X})=0$. We have a symmetric bilinear form $(*,**)$
on $H^{2}(X,\mZ)$ by the cup product. Then $(H^{2}(X,\mZ),(*,**))$
is an even unimodular lattice of signature $(3,19)$, which is isomorphic
to $L_{K3}:=E_{8}(-1)^{\oplus2}\oplus U^{\oplus3}$ where $U$ is
the hyperbolic lattice $(\mZ\oplus\mZ,\left(\begin{smallmatrix}0 & 1\\
1 & 0\end{smallmatrix}\right)$). Denote by $NS_{X}=Pic(X)$ the Picard (N\'eron-Severi) lattice
and set $\rho(X)=\rk NS_{X}$. $NS_{X}$ is the primitive sub-lattice
in $H^{2}(X,\mZ)$ and has signature $(1,\rho(X)-1)$. The orthogonal
complement $T_{X}=(NS_{X})^{\perp}$ in $H^{2}(X,\mZ)$ is called
transcendental lattice. $T_{X}$ has signature $(2,20-\rho(X))$.
The extension $\tilde{H}(X,\mZ)=H^{0}(X,\mZ)\oplus H^{2}(X,\mZ)\oplus H^{4}(X,\mZ)\simeq E_{8}(-1)^{\oplus2}\oplus U^{\oplus4}$
is called Mukai lattice.

Let us denote by $\omega_{X}$ the nowhere vanishing holomorphic two
form of $X$ which is unique up to constant. Then the Global Torelli
theorem says that K3 surfaces $X$ and $X'$ are isomorphic iff there
exists a Hodge isometry, i.e., a lattice isomorphism $\varphi:H^{2}(X,\mZ)\to H^{2}(X',\mZ$)
which satisfies $\varphi(\mC\omega_{X})=\mC\omega_{X'}$. Extending
earlier works by Mukai \cite{Mukai1} in 80', Orlov \cite{Orlov}
has formulated a similar Global Torelli theorem for the derived categories
of coherent sheaves on K3 surfaces: 
\begin{thm}[\cite{Mukai1}\cite{Orlov}]
 K3 surfaces $X$ and $X'$ are derived equivalent, $D^{b}(X)\simeq D^{b}(X'),$
if and only if there exists a Hodge isometry of transcendental lattices
$(T_{X},\mC\omega_{X})\simeq(T_{X'},\mC\omega_{X'})$. 
\end{thm}
Due to the uniqueness theorem of primitive embeddings into indefinite
lattices (see Theorem \ref{thm:Nikulin1} in Appendix), we note that
the Hodge isometry $(T_{X},\mC\omega_{X})\simeq(T_{X'},\mC\omega_{X'})$
above always extends to that of the Mukai lattice $(\tilde{H}(X,\mZ),\mC\omega_{X})\simeq(\tilde{H}(X',\mZ),\mC\omega_{X'})$,
and hence we can rephrase the above theorem in terms of the Hodge
isometry of Mukai lattices. 

Consider smooth projective varieties $X$ and $Y$. $Y$ is called
Fourier-Mukai partner of $X$ if $D^{b}(Y)\simeq D^{b}(X)$. We denote
the set of Fourier-Mukai partners (up to isomorphisms) of $X$ by
\[
FM(X)=\left\{ Y\mid D^{b}(Y)\simeq D^{b}(X)\right\} /\text{isom.}\]
For a K3 surface $X$, the set $FM(X)$ consists of K3 surfaces (see
\cite[Cor.10.2]{Huy} for example) and its cardinality is known to
be finite, i.e. $|FM(X)|<\infty$ in \cite{BM}. Studying all possible
obstructions for extending a Hodge isometry $(T_{X},\mC\omega_{X})\simeq(T_{X'},\mC\omega_{X'})$
between the transcendental lattices to the corresponding Hodge isometry
$(H^{2}(X,\mZ),\mC\omega_{X})\simeq(H^{2}(Y,\mZ),\mC\omega_{Y})$,
the following counting formula has been obtained:
\begin{thm}[\cite{HLOYcounting}]
\label{thm:HOLY-counting-F}For a K3 surface $X$, we have \[
|FM(X)|=\sum_{\mathcal{G}(NS_{X})=\left\{ S_{1},..,S_{N}\right\} }\big\vert O(S_{i})\diagdown O(A_{S_{i}})\diagup O_{Hodge}(T_{X},\mC\omega_{X})\big\vert,\]
where $\mathcal{G}(NS_{X})$ is the isogeny classes of the lattice
$NS_{X}$, $A_{S_{i}}=(S^{*}/S,q:S^{*}/S\to\mQ/\mZ)$ is the discriminant
of the lattice $S_{i}$, and $O(S_{i})$ and $O(A_{S_{i}})$ are isometries
of $S_{i}$ and $A_{S_{i}}$. $O_{Hodge}(T_{X},\mC\omega_{X})$ is
the Hodge isometries of $(T_{X},\mC\omega_{X})$. 
\end{thm}
We refer to \cite{HLOYcounting} for the details (see also \cite{HP}).
Since the isogeny classes of a lattice are finite, the counting formula
contains the earlier result $|FM(X)|<\infty$. When $X$ is a K3 surface
with $\rho(X)=1$ and deg$(X)=2n$, the counting formula coincides
with the result in \cite{Og} (obtained by counting the so-called
over-lattices); \begin{equation}
|FM(X)|=2^{p(n)-1}\,(=\frac{1}{2}|O(A_{NS_{X}})|),\label{eq:FMone}\end{equation}
where $p(n)$ is the number of prime factors of $n$ (we set $p(1)=1$).
In fact, much is known by \cite{Mukai3} in this case that we have\[
FM(X)=\left\{ \mathcal{M}_{X}(r,h,s)\mid n=rs,(r,s)=1\right\} ,\]
in terms of the moduli space of stable vector bundles $\sE$ on $X$
with Mukai vector $(r,h,s)=$$ch(\sE)\sqrt{Td_{X}}$ in $H^{0}(X,\mZ)\oplus H^{2}(X,\mZ)\oplus H^{4}(X,\mZ$)
(see also \cite{HLOYpicOne}). We will study in detail the first non-trivial
example of $|FM(X)|\not=1$ ($n=6$) in Subsection \ref{sub:An-example-due-Mukai}. 

\vskip0.5cm

\subsection{Marked $M$-polarized K3 surfaces\label{sub:M-polarized-K3-surfaces}}

A K3 surface $X$ with a choice of isomorphism $\phi:H^{2}(X,\mZ)\simarrow L_{K3}$
is called a marked K3 surface $(X,\phi)$. Marked K3 surfaces $(X,\phi)$
and $(X',\phi')$ are isomorphic if there exists an isomorphism $f:X\to X'$
satisfying $\phi'=\phi\circ f^{*}$. By the Global Torelli theorem,
$(X,\phi)$ and $(X',\phi')$ are isomorphic iff there exists a Hodge
isometry $\varphi:(H^{2}(X',\mZ),\mC\omega_{X'})\simarrow(H^{2}(X,\mZ),\mC\omega_{X})$
such that $\phi'=\phi\circ\varphi$ (see \cite{BHPv} for more details
of K3 surfaces). 

Consider a lattice $M$ of signature $(1,t)$ and fix a primitive
embedding $i:M\hookrightarrow L_{K3}$. A marked K3 surface $(X,\phi)$
is called marked $M$-polarized K3 surface if $\phi^{-1}(M)\subset NS_{X}$
(where we write $\phi^{-1}(M)=(\phi^{-1}\circ i)(M)$ for short).
Marked $M$-polarized K3 surfaces $(X,\phi)$ and $(X',\phi')$ are
isomorphic if there exists a lattice isomorphism $\varphi:L_{K3}\simarrow L_{K3}$
such that \begin{equation}
\begin{matrix}\xymatrix{H^{2}(X,\mZ)\ar_{\sim}^{\phi}[r] & L_{K3}\ar_{\varphi}^{\wr}[d] & M\ar@{_{(}->}_{i}[l]\ar@{=}[d]\\
H^{2}(X',\mZ)\ar_{\sim}^{\phi'}[r] & L_{K3} & M\ar@{_{(}->}_{i}[l]}
\end{matrix}\label{eq:iso-M-pol}\end{equation}
and the composition $(\phi')^{-1}\circ\varphi\circ\phi:(H^{2}(X,\mZ),\mC\omega_{X})\to(H^{2}(X',\mZ),\mC\omega_{X'})$
is a Hodge isometry. The lattice isomorphism $\varphi$ in (\ref{eq:iso-M-pol})
is an element of the group \[
\Gamma(M)=\left\{ g\in O(L_{K3})\mid g(m)=m\;(\forall m\in M)\right\} .\]
Consider the orthogonal lattice $M^{\perp}=(i(M))^{\perp}$. Then
there is a natural injective homomorphism $\Gamma(M)\to O(M^{\perp}).$
The image is known to be described by the kernel $O(M^{\perp})^{*}:=\Ker\left\{ O(M^{\perp})\to O(A_{M^{\perp}})\right\} $
of the natural homomorphism to the isometries of the discriminant
$A_{M^{\perp}}$ (see \cite[Prop.3.3]{Dolgavhev}). 

A marked K3 surfaces $(X,\phi)$ determines the period points $\phi(\mC\omega_{X})$
in the period domain $\sD$=$\left\{ [\omega]\in\mP(L_{K3}\otimes\mC)\mid(\omega,\omega)=0,(\omega,\bar{\omega})>0\right\} $.
By the surjectivity of the period map, $\sD$ gives a classifying
space of the (not necessarily projective) marked K3 surfaces. Then,
by the Global Torelli theorem, the quotient $\sD/O(L_{K3})$ classifies
the isomorphism classes of (not necessarily projective) marked K3
surfaces. 

From the definition, it is easy to deduce that marked $M$-polarized
K3 surfaces are classified by the period points in the following domain\[
\sD(M^{\perp}):=\left\{ [\omega]\in\mP(M^{\perp}\otimes\mC)\mid(\omega,\omega)=0,(\omega,\bar{\omega})>0\right\} ,\]
which has two connected components $\sD(M^{\perp})=\sD(M^{\perp})^{+}\sqcup\sD(M^{\perp})^{-}$.
Let us define $O^{+}(M^{\perp})\subset O(M^{\perp})$ to be the isometries
of $M^{\perp}$ which preserve the orientations of all positive two
spaces in $M^{\perp}\otimes\mathbb{R}$. Then the isomorphisms classes
of marked $M$-polarized K3 surfaces are classified by the following
quotient,\begin{equation}
\sD(M^{\perp})/O(M^{\perp})^{*}\simeq\sD(M^{\perp})^{+}/O^{+}(M^{\perp})^{*}(\simeq\sD(M^{\perp})^{-}/O^{+}(M^{\perp})^{*}),\label{eq:Domain-M}\end{equation}
where $O^{+}(M^{\perp})^{*}:=O^{+}(M^{\perp})\cap O(M^{\perp})^{*}$
is the \textit{monodromy} group which acts on the period points $\phi(\mC\omega_{X})\in\sD(M^{\perp})^{\pm}$
of marked $M$-polarized K3 surfaces $(X,\phi$). 

\vskip0.5cm

\subsection{$M$-polarizable K3 surfaces }

Let us fix a primitive lattice embedding $i:M\hookrightarrow L_{K3}$
as in the preceding subsection. Following \cite{HLOYauteq}, we call
a K3 surface $X$ $M$-\textit{polarizable} if there is a marking
$\phi:H^{2}(X,\mZ)\overset{\sim}{\to}L_{K3}$ such that $(\phi^{-1}\circ i)(M)\subset NS_{X}$.
Two $M$-polarizable K3 surfaces $X$ and $X'$ are defined to be
isomorphic if there exists lattice isomorphisms $\varphi:L_{K3}\simarrow L_{K3}$
and $g:M\simarrow M$ which make the following diagram commutative:
\begin{equation}
\begin{matrix}\xymatrix{H^{2}(X,\mZ)\ar_{\sim}^{\phi}[r] & L_{K3}\ar_{\varphi}^{\wr}[d] & M\ar@{_{(}->}_{i}[l]\ar_{g}^{\wr}[d]\\
H^{2}(X',\mZ)\ar_{\sim}^{\phi'}[r] & L_{K3} & M\ar@{_{(}->}_{i}[l]}
\end{matrix}\label{eq:iso-M-polarizable}\end{equation}
and the composition $(\phi')^{-1}\circ\varphi\circ\phi:(H^{2}(X,\mZ),\mC\omega_{X})\to(H^{2}(X',\mZ),\mC\omega_{X'})$
is a Hodge isometry. Note that, as we see in the diagram, the definition
of the isomorphism is slightly generalized for the $M$-polarizable
K3 surfaces. Hence, although $M$-polarizable K3 surfaces $X$ are
obtained by forgetting the marking $\phi$ from the marked $M$-polarized
K3 surfaces $(X,\phi)$, their isomorphism classes are possibly different.
We saw in the last subsection that the isomorphism classes of marked
$M$-polarized K3 surfaces are classified by the quotient $\sD(M^{\perp})/O(M^{\perp})^{*}$.
On the other hand, the classifying space of the isomorphism classes
of $M$-polarizable K3 surfaces is given by a similar quotient of
$\sD(M^{\perp})$ but with a group which resides between $O(M^{\perp})^{*}$
and $O(M^{\perp})$. 

\vskip0.5cm

\subsection{Mirror symmetry of K3 surfaces\label{sub:Mirror-symmetry-of-K3}}

In \cite{Dolgavhev}, Dolgachev defined mirror symmetry of marked
$M$-polarized K3 surfaces. To summarize his construction/definition,
let us fix a primitive embedding $i:M\hookrightarrow L_{K3}$ of a
lattice $M$ of signature $(1,t)$ and \textit{assume} that the orthogonal
lattice $M^{\perp}$ has a decomposition $M^{\perp}=\check{M}\oplus U$,
i.e. \[
M\oplus M^{\perp}=M\oplus U\oplus\check{M}\subset L_{K3},\]
where $U$ is the hyperbolic lattice. Since the signature of $\check{M}$
is $(1,\check{t})=(1,19-t)$, the primitive embedding $i:\check{M}\hookrightarrow L_{K3}$
naturally introduces marked $\check{M}$-polarized K3 surfaces. Marked
$\check{M}$-polarized K3 surfaces are classified by $\sD(\check{M}^{\perp})$,
while marked $M$-polarized K3 surfaces are classified by $\sD(M^{\perp})$. 

For a general marked $M$-polarized K3 surface $(X,\phi)$ and a general
marked $\check{M}$-polarized K3 surface $(\check{X},\check{\phi})$,
we have the following isomorphisms:\begin{equation}
NS_{X}\simeq M,\; T_{X}\simeq U\oplus\check{M};\;\; NS_{\check{X}}\simeq\check{M},\; T_{\check{X}}\simeq U\oplus M,\;\label{eq:lattice-mirror-symmetry}\end{equation}
and observe the exchange of the algebraic and transcendental cycles
(up to the factor $U$). This exchange is the hallmark of the mirror
symmetry of K3 surfaces. Also we see the so-called {}``mirror map''
\cite{L-Y} for K3 surfaces in the following isomorphisms (see e.g.
\cite[Prop.1]{GrossWilson}):\begin{equation}
V(M)\simeq\sD(\check{M}^{\perp}),\; V(\check{M})\simeq\sD(M^{\perp}),\label{eq:mirror-map-V-D}\end{equation}
where $V(M)$ is the tube domain defined by $V(M)=\left\{ B+i\kappa\in M\otimes\mC\mid(\kappa,\kappa)>0\right\} $
and similar definition for $V(\check{M})$. $V(M)$ and $V(\check{M})$
are regarded as the tube domains for the complexified K\"ahler moduli
spaces of $(X,\phi)$ and $(\check{X},\check{\phi})$, respectively,
and hence (\ref{eq:mirror-map-V-D}) describes the mirror isomorphisms
between the complex structure and (complexified) K\"ahler moduli
spaces. There are several different ways to define mirror symmetry
of K3 surfaces \cite{Ba1,SYZ}. See references \cite{GrossWilson,Bel},
for example, for the relations among them. 

\vskip0.5cm

\subsection{Homological mirror symmetry }

There is a slight asymmetry in the exchange of the Picard lattices
and the transcendental lattices in (\ref{eq:lattice-mirror-symmetry}).
This can be remedied by considering the (numerical) Grothendieck group
together with a (non-degenerate) pairing $([\sE],[\mathcal{F}])=-\chi(\sE,\mathcal{F})$
where $\chi(\sE,\mathcal{F})=\sum(-1)^{i}\dim\Ext_{\sO_{X}}^{i}(\sE,\mathcal{F})$.
Namely, we understand the isomorphisms (\ref{eq:lattice-mirror-symmetry})
as \begin{equation}
T_{\check{X}}\simeq U\oplus M\simeq(K(X),(*,**)),\;\; T_{X}\simeq U\oplus\check{M}\simeq(K(\check{X}),(*,**)).\label{eq:Hom-Mirror-K}\end{equation}
Note that the form $(*,**)$ is symmetric due to the Serre duality
for K3 surfaces. Also we note that $K(X)$ contains $[\sO_{x}]$ and
$-[\sI_{x}]$, in addition to $[\sO_{D}]=[\sO_{X}]-[\sO_{X}(-D)]$
for $D\in Pic(X)$ (likewise for $K(\check{X})$). By Riemann-Roch
theorem, it is easy to see that $[\sO_{x}]$ and $-[\sI_{x}]$ explain
the additional factor $U$ in $U\oplus M$. The above isomorphisms
are consequences of the homological mirror symmetry due to Kontsevich
\cite{Ko}, but we refrain from going into the details about this
in this article. 

\vskip0.5cm

\subsection{\label{sub:FM(X)-and-mirror}FM(X) and mirror symmetry }

Let us consider the case $M_{n}=\langle2n\rangle$, i.e., $(\mZ h,h^{2}=2n)$
in detail. We first note that we can embed the lattice $M_{n}$ into
the hyperpolic lattice $U$ by making a primitive embedding $\langle2n\rangle\oplus\langle-2n\rangle\subset U$.
Then, since primitive embedding $i:M_{n}\hookrightarrow L_{K3}$ is
unique up to isomorphism due to Theorem \ref{thm:Nikulin2}, we may
assume that the embedding $i:M_{n}\hookrightarrow L_{K3}$ is given
by \[
M_{n}\oplus M_{n}^{\perp}=\langle2n\rangle\oplus(U\oplus\check{M}_{n})\subset L_{K3}\]
where $M_{n}^{\perp}:=(i(M_{n}))^{\perp}=\langle-2n\rangle\oplus U^{\oplus2}\oplus E_{8}(-1)^{\oplus2}$
is the orthogonal lattice and $\check{M}_{n}:=\langle-2n\rangle\oplus U\oplus E_{8}(-1)^{\oplus2}$. 

Let $(X,\phi)$ be a marked $M_{n}$-polarized K3 surface, and $h$
be its polarization $(h^{2}=2n)$. Then we have $|FM(X)|=2^{p(n)-1}$
from the counting formula. On the other hand, for a general marked
$\check{M}_{n}$-polarized K3 surface $(\check{X},\check{\phi})$,
we have $|FM(\check{X})|=1$ since $\rho(\check{X})=19$ and $A_{\check{M}_{n}}\simeq\mZ/2n\mZ$
(see \cite[Cor.2.6]{HLOYcounting} and also \cite[Proposition 6.2]{Mukai1}). 

It has been argued in \cite{HLOYauteq} that the number $|FM(X)|=2^{p(n)-1}$
has a nice interpretation from the monodromy group which acts on the
period domain $\sD(\check{M}_{n}^{\perp})^{+}$ for the mirror marked
polarized $\check{M}_{n}$-polarized K3 surfaces. Roughly speaking,
the number $\vert FM(X)\vert$ appears as the covering degree of the
map from $\sD(\check{M}_{n}^{\perp})^{+}/O^{+}(\check{M}_{n}^{\perp})^{*}$
to the corresponding quotient for the isomorphism classes of $\check{M}_{n}$-polarizable
K3 surfaces. 

We have determined, in Subsection \ref{sub:M-polarized-K3-surfaces},
the monodromy group of the marked $\check{M}_{n}$-polarized K3 surfaces
by $O^{+}(\check{M}_{n}^{\perp})^{*}=O^{+}(\check{M}_{n}^{\perp})\cap O(\check{M}_{n}^{\perp})^{*}$.
As for the $\check{M}_{n}$-polarizable K3 surfaces, the corresponding
group becomes larger. 
\begin{lem}[{\cite[Lem.1.14, Def.1.15]{HLOYauteq}}]
 \label{lem:-The-monodromy}The monodromy group of the $\check{M}_{n}$-polarizable
K3 surfaces is given by $O^{+}(\check{M}_{n}^{\perp})$/$\left\{ \pm id\right\} $. 
\end{lem}
\noindent By definition, for $\check{M}_{n}$-polarizable K3 surfaces
$\check{X},\check{X}',$ we have markings $\phi,\phi'$ such that
$(\check{X},\check{\phi})$ and $(\check{X'},\check{\phi'})$ are
marked $\check{M}_{n}$-polarized K3 surfaces. Then, the above lemma
can be deduced from the following diagram which describes the isomorphism
of $\check{M}_{n}$-polarizable K3 surfaces: \begin{equation}
\begin{matrix}\xymatrix{H^{2}(\check{X},\mZ)\ar_{\sim}^{\check{\phi}}[r] & L_{K3}\ar_{\varphi}^{\wr}[d] & \check{M}_{n}\ar@{_{(}->}_{i}[l]\ar_{g}^{\wr}[d]\\
H^{2}(\check{X'},\mZ)\ar_{\sim}^{\check{\phi'}}[r] & L_{K3} & \check{M}_{n}\ar@{_{(}->}_{i}[l]}
\end{matrix}\label{eq:xxxg}\end{equation}
Here we sketch the proof of the lemma: Suppose an element $h\in O(\check{M}_{n}^{\perp})$
is given. Since primitive embedding $\check{M}_{n}^{\perp}=U\oplus M_{n}\hookrightarrow L_{K3}$
is unique by Theorem \ref{thm:Nikulin2}, $h$ extends to an isomorphism
$\varphi:L_{K3}\to L_{K3}$ and also determines an isomorphism $g:\check{M}_{n}\to\check{M}_{n}$
on the orthogonal complement of $\check{M}_{n}^{\perp}$. By the surjectivity
of the period map, we see that $\varphi$ extends to an isomorphism
of $\check{M}_{n}$-polarizable K3 surfaces. From the relation $\sD(\check{M}_{n}^{\perp})/O(\check{M}_{n}^{\perp})\simeq\sD(\check{M}_{n}^{\perp})^{+}/O^{+}(\check{M}_{n}^{\perp})$
and the fact that $\left\{ \pm id\right\} $ has a trivial action
on $\sD(\check{M}_{n}^{\perp})^{+}$, the group $O^{+}(\check{M}_{n}^{\perp})/\left\{ \pm id\right\} $
identifies the $\check{M}_{n}$-polarizable K3 surfaces which are
isomorphic to each other. In this sense, we can call the quotient
group $O^{+}(\check{M}_{n}^{\perp})/\left\{ \pm id\right\} $ the
monodromy group of $\check{M}_{n}$-polarizable K3 surfaces. \hfill $\square$ 

Now we can see the FM number $|FM(X)|=2^{p(n)-1}$ as the covering
degree of the map \[
\sD(\check{M}_{n}^{\perp})^{+}/O^{+}(\check{M}_{n}^{\perp})^{*}\to\sD(\check{M}_{n}^{\perp})^{+}/O^{+}(\check{M}_{n}^{\perp}),\]
which we evaluate for $n\not=1$ (see \cite[Theorem 1.18]{HLOYauteq}
for details) as \[
[O^{+}(\check{M}_{n}^{\perp})/\left\{ \pm id\right\} :O^{+}(\check{M}_{n}^{\perp})^{*}]=2^{p(n)-1},\]
where we recall the fact that $\left\{ \pm id\right\} $ acts trivially
on the domain. The covering degree can be explained by the nontrivial
actions of $g$ in the diagram (\ref{eq:xxxg}), which implies that
$(\check{X},\check{\phi})$ and $(\check{X'},\check{\phi'})$ are
related by Hodge isometries that have non-trivial actions on the Picard
lattice. The monodromy group $O^{+}(\check{M}_{n}^{\perp})^{*}$ comes
from the Dehn twists which preserve (the cohomology classes of) generic
symplectic forms (K\"ahler forms) $\kappa_{\check{X}}$ (\cite[Thm.1.9]{HLOYauteq}).
Then the covering group represents isomorphisms of K3 surfaces which
do not preserve the (cohomology classes of) generic symplectic forms
$\kappa_{\check{X}}$. This is the mirror symmetry interpretation
of $FM(X)$ made in {[}ibid{]}, where the relation of the Dehn twits
to $Auteq\, D^{b}(X)$ has been discussed in more detail. 

\vskip0.5cm

\subsection{An example due to Mukai\label{sub:An-example-due-Mukai}}

Here we consider an explicit construction of the $M_{6}=\langle12\rangle$-polarized
K3 surfaces due to Mukai \cite{Mukai4}. We see general properties
discussed in the last subsections for this specific example, and make
an observation that will be shared with the examples of Calabi-Yau
threefolds in the subsequent sections. Note that $FM(X)=\left\{ X,Y\right\} $
with $Y\simeq\sM_{X}(2,h,3)$ for general $M_{6}$-polarized K3 surfaces
$X$.

\subsubsection{\label{sub:OG(5,10)-def}{\itshape\bfseries  Linear sections of
$\mathbf{OG(5,10)}$}}

Let us consider orthogonal Grassmannian $\mathrm{OG}(5,10)$ which
parametrizes maximal isotropic subspaces of $\mC^{10}$ with a fixed
non-degenerate quadratic form. $\mathrm{OG}(5,10)$ has two connected
components $\mathrm{OG}^{\pm}(5,10)$, which are isomorphic to each
other. $\mathrm{OG}^{+}(5,10)\simeq\mathrm{OG}^{-}(5,10)$ is called
spinor variety $\mathbb{S}_{5}$ (of dimension $10$), and can be
embedded into the projective space $\mP(S_{16})$ of the spin representation
of $SO(10)$. $\mathrm{OG}^{+}(5,10)$ is the Hermitian symmetric
space $\mathrm{SO}(10,\mathbb{R})/\mathrm{U}(5)$, and its Picard
group is generated by the ample class of the above spinor embedding.
The projective dual variety (discriminantal variety) $\mathbb{S}_{5}^{*}$
in the dual projective space $\mP(S_{16}^{*})$ is known to be isomorphic
to $\mathbb{S}_{5}$. Mukai \cite{Mukai4} constructed a smooth K3
surface of degree $12$ (with Picard group $\mZ h$) by considering
a complete linear section $X=\mathbb{S}_{5}\cap H_{1}\cap...\cap H_{8}$
and observed that the moduli space of stable vector bundles $\mathcal{M}_{X}(2,h,3)$
over $X$ is isomorphic to a K3 surface $Y$, which is defined in
the dual variety $\mathbb{S}_{5}^{*}$ in the following way: Let $L_{8}$
be a general 8-dimensional linear subspace in $S_{16}^{*}$ and by
$L_{8}^{\perp}$ its orthogonal space in $S_{16}$. Then the K3 surfaces
$X$ and $Y$ above are given by the {}``orthogonal linear sections
to each other'',\[
\begin{aligned}X=\mathbb{S}_{5}\cap\mP(L_{8}^{\perp})\subset\mP(S_{16}), & \;\;\; Y=\mathbb{S}_{5}^{*}\cap\mP(L_{8})\subset\mP(S_{16}^{*}).\end{aligned}
\]

Due to the isomorphism $Y\simeq\mathcal{M}_{X}(2,h,3)$ (see \cite{IM}
for a proof), we can write the equivalence $\Phi_{\sP}:D^{b}(Y)\simeq D^{b}(X)$
using the universal bundle $\sP$ over $X\times Y$ as the kernel
of the Fourier-Mukai transform $\Phi_{\sP}(-)=R\pi_{X*}(L\pi_{Y}^{*}(-)\otimes\sP$).

\subsubsection{{\itshape\bfseries Mirror family of $M_{6}$-polarized K3 surfaces} }

Let us consider marked $\check{M}_{6}$-polarized K3 surfaces, which
are the mirror K3 surfaces of $X$ as defined in Subsection \ref{sub:Mirror-symmetry-of-K3}.
Their isomorphism classes are classified by the points on the quotient
of the period domain $\sD(\check{M}_{6}^{\perp})$ by the group $O(\check{M}_{6}^{\perp})^{*}$.
Noting that $\sD(\check{M}_{6}^{\perp})\simeq V(M_{6})$ consists
of two copies of the upper half pane $\mathbb{H}_{+}$ and an isomorphism
$O^{+}(\check{M}_{6}^{\perp})^{*}\simeq\Gamma_{0}(6)_{+6}$ (see \cite[Thm.(7.1), Rem.(7.2)]{Dolgavhev}),
we have \[
\sD(\check{M}_{6}^{\perp})^{+}/O^{+}(\check{M}_{6}^{\perp})^{*}\simeq\mathbb{H}_{+}/\Gamma_{0}(6)_{+6},\]
see Fig.1. On the other hand, we have an isomorphism $O^{+}(\check{M}_{6}^{\perp})/\left\{ \pm id\right\} \simeq\Gamma_{0}(6)_{+}$
for the monodromy group of the $\check{M}_{6}^{\perp}$-polarizable
K3 surfaces {[}ibid{]} (see also \cite[Thm.5.5]{HLOYauteq}). For
these two groups, we have the following presentations: \begin{equation}
\begin{matrix}\Gamma_{0}(6)_{+}=\langle\left(\begin{matrix}1 & 1\\
0 & 1\end{matrix}\right),\left(\begin{smallmatrix}0 & -\frac{1}{\sqrt{6}}\\
\sqrt{6} & 0\end{smallmatrix}\right),\left(\begin{smallmatrix}\sqrt{3} & \frac{1}{\sqrt{3}}\\
2\sqrt{3} & \sqrt{3}\end{smallmatrix}\right)\rangle=:\langle T_{0},S_{1},S_{2}S_{1}\rangle, & \quad\\
\Gamma_{0}(6)_{+6}=\langle\left(\begin{matrix}1 & 1\\
0 & 1\end{matrix}\right),\left(\begin{smallmatrix}0 & -\frac{1}{\sqrt{6}}\\
\sqrt{6} & 0\end{smallmatrix}\right),\left(\begin{matrix}5 & 2\\
12 & 5\end{matrix}\right)\rangle=:\langle T_{0},S_{1},(S_{2}S_{1})^{2}\rangle,\end{matrix}\label{eq:Gamma6}\end{equation}
with $S_{2}=\left(\begin{smallmatrix}-\sqrt{2} & \frac{1}{\sqrt{2}}\\
-3\sqrt{2} & \sqrt{2}\end{smallmatrix}\right)$. Explicit relations of $\Gamma_{0}(6)_{+}$ and $\Gamma_{0}(6)_{+6}$
to $O^{+}(\check{M}_{6}^{\perp})/\left\{ \pm\id\right\} $ and $O^{+}(\check{M}_{6}^{\perp})^{*}$,
respectively, are given by fixing an isomorphism $\check{M}_{6}^{\perp}\simeq(\mZ^{\oplus3},\Sigma_{6})$
with $\Sigma_{6}=\left(\begin{smallmatrix}0 & 0 & 1\\
0 & 12 & 0\\
1 & 0 & 0\end{smallmatrix}\right)$ and an anti-homomorphism $R:PSL(2,\mathbb{R})\to SO(2,1,\mathbb{R})$,\[
R:\left(\begin{smallmatrix}a & b\\
c & d\end{smallmatrix}\right)\mapsto\left(\begin{smallmatrix}a^{2} & -2ac & -\frac{c^{2}}{6}\\
-ab & ad+bc & \frac{cd}{6}\\
-6b^{2} & 12bd & d^{2}\end{smallmatrix}\right)\in SO(2,1,\mathbb{R}),\]
where $SO(2,1,\mathbb{R})=\left\{ g\in\mathrm{Mat(}3,\mathbb{R})\mid\,^{t}g\Sigma_{6}g=\Sigma_{6}\right\} $.
Here, we naturally consider $O^{+}(\check{M}_{6}^{\perp}),O^{+}(\check{M}_{6}^{\perp})^{*}$
in $SO(2,1,\mathbb{R})$ (and the image of $O^{+}(\check{M}_{6}^{\perp})\to SO(2,1,\mathbb{R}),$
$g\mapsto(\det g)g$ for $O^{+}(\check{M}_{6}^{\perp})/\left\{ \pm\id\right\} $).
The group index $[\Gamma_{0}(6)_{+}:\Gamma_{0}(6)_{+6}]=2$ is obvious
from (\ref{eq:Gamma6}) and this is the mirror interpretation of $|FM(X)|=2$
in this case.

We can actually construct a family of (marked) $\check{M}_{6}$-polarized
K3 surfaces $\check{\sX}=\left\{ \check{X}_{x}\right\} _{x\in\mP^{1}}$
parametrized by $\mP^{1}$ \cite{L-Y,PetersStienstra}, whose Picard-Fuchs
differential equation for period integrals has the following form
with $\theta_{x}=x\frac{d\;}{dx}$: \begin{equation}
\left\{ \theta_{x}^{3}-x(2\theta_{x}+1)(17\theta_{x}^{2}+17\theta_{x}+5)+x^{2}(\theta_{x}+1)^{3}\right\} \omega(x)=0,\label{eq:PF-K3}\end{equation}
where $\omega(x)=\int_{\gamma}\Omega(\check{X}_{x})$ is the period
integrals of nowhere vanishing holomorphic 2 form $\omega_{\check{X}_{x}}=\Omega(\check{X}_{x})$
with respect to a transcendental cycle $\gamma\in H_{2}(\check{X}_{x_{0}},\mZ$).
In \cite{HLOYauteq}, the corresponding $\mP^{1}$ family of $\check{M}_{6}$-polarizable
K3 surfaces has been studied in detail. \def\figIEPS{\resizebox{8cm}{!}{\includegraphics{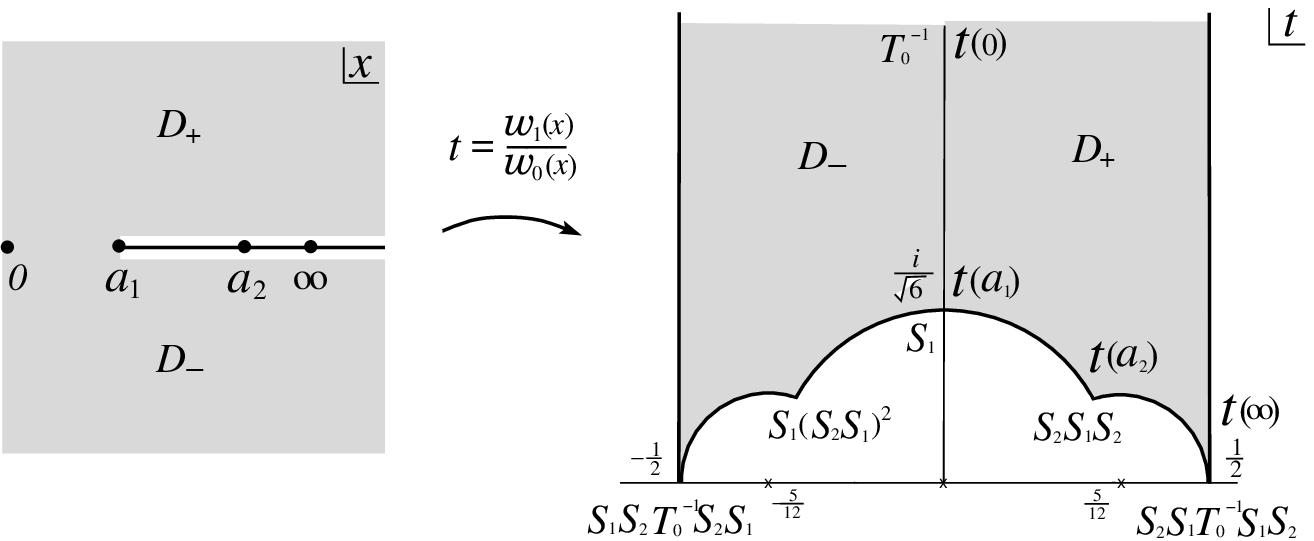}}} 
\def\FigI{
\begin{xy} (0,0)*{\figIEPS},
\end{xy}}

\begin{figure}
\begin{centering}
\includegraphics[scale=0.7]{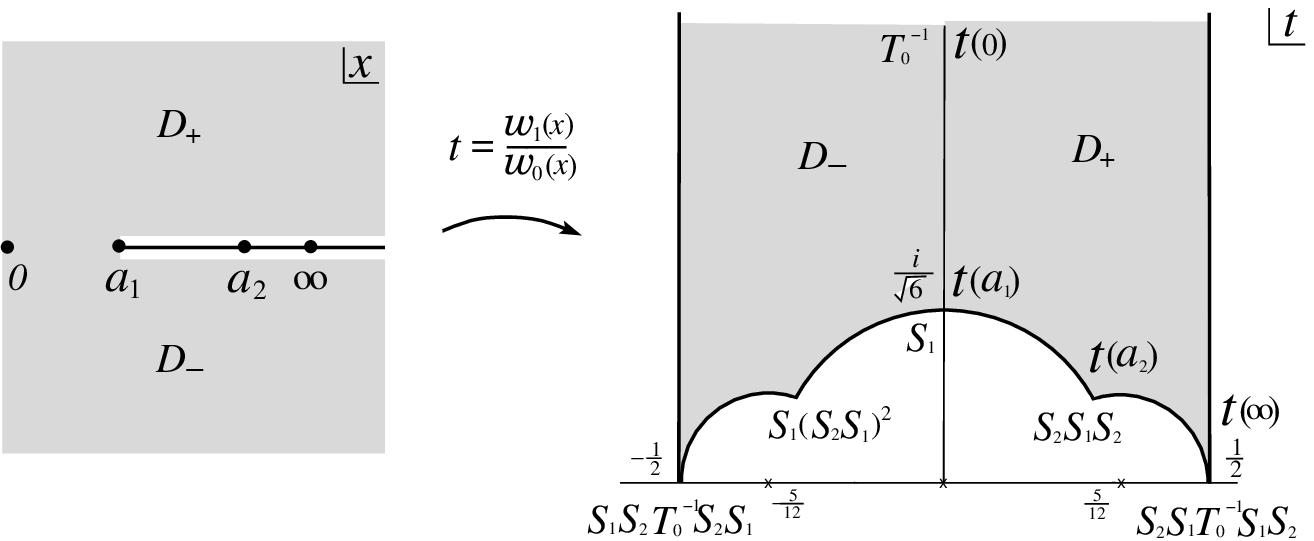}
\par\end{centering}

\begin{fcaption}\item{Fig.2.1.}\label{Fig:fig1} \textbf{Fundamental
Domain of $\Gamma_{0}(6)_{+6}$.} The image of the mirror map $t=t(x)$
\cite{L-Y} is depicted as the fundamental domain of $\Gamma_{0}(6)_{+6}$.
The images of $0,a_{1},a_{2},\infty$ have nontrivial isotropy groups,
which explain the monodromy around each point. The generators of the
isotropy groups are shown at each point. \end{fcaption}
\end{figure}

\subsubsection{\label{sub:Monodromy-calc-K3}{\itshape\bfseries Monodromy calculations
}}

As we see in Fig.1, there are two cusps in $\mathbb{H}_{+}/\Gamma_{0}(6)_{+6}$.
By Proposition \ref{pro:Monodromy-K3-deg12} below, we see that these
two are identified by the action of an element $\Gamma_{0}(6)_{+}\setminus\Gamma_{0}(6)_{+6}$.
In fact, these cusps correspond to the maximally unipotent monodromy
(MUM) points at $x=0$ and $x=\infty$ of (\ref{eq:PF-K3}), which
we read in the following Riemann's $\sP$ scheme:\[ \left\{ \small\begin{array}{cccc} 0 & a_{1} & a_{2} & \infty \\ 
\hline  
0 & 0 & 0 & 1\\ 0 & 1 & 1 & 1\\ 0 & \frac{1}{2} & \frac{1}{2} & 1
\end{array}  \right\}  \]with $a_{1}:=17-12\sqrt{2},a_{2}:=17+12\sqrt{2}$ (see \cite{Morrison}
for a general definition of MUM points). The relation of these cusps
becomes explicit by constructing an integral basis of the solutions
of the Picard-Fuchs equation (\ref{eq:PF-K3}) which is compatible
with the mirror isomorphism $T_{\check{X}}\simeq(K(X),-\chi(*,**))$
in (\ref{eq:Hom-Mirror-K}). Since the construction is general for
other K3 surfaces \cite{HoIIa} and also parallel to that for Calabi-Yau
threefolds (see \cite[Secti.2]{HoTa1}), we briefly sketch it here.
Firstly, we set up the local solutions about the MUM point $x=0$
of the form $w_{0}(x)=1+O(x)$ and\[
\begin{matrix}\begin{aligned}w_{1}(x)= & w_{0}(x)\log(x)+w_{1}^{reg}(x),\\
w_{2}(x)= & -w_{0}(x)(\log x)^{2}+2w_{1}(x)\log x+w_{2}^{reg}(x)\end{aligned}
\end{matrix}\]
 requiring the forms $w_{1}^{reg}(x)=c_{1}x+O(x)$ and $w_{2}^{reg}(x)=c_{2}x^{2}+O(x^{2}).$
We make similar solutions $\tilde{w}_{k}(z)$ $(z=\frac{1}{x})$ around
$z=0$ requiring $\tilde{w}_{0}(z)=z(1+O(z)),$$\tilde{w}{}_{1}^{reg}(z)=z(\tilde{c}_{1}z+O(z))$
and $\tilde{w}{}_{2}^{reg}(z)=z(\tilde{c}_{2}z^{2}+O(z^{2}))$. Using
these, we set the following ansatz for the integral basis:\begin{equation}
\Pi(x)=N_{x}\left(\begin{smallmatrix}1 & 0 & 0\\
0 & 1 & 0\\
0 & 0 & -\frac{\mathrm{deg}X}{2}\end{smallmatrix}\right)\left(\begin{smallmatrix}n_{0}w_{0}\\
n_{1}w_{1}\\
n_{2}w_{2}\end{smallmatrix}\right),\;\;\tilde{\Pi}(z)=N_{z}\left(\begin{smallmatrix}1 & 0 & 0\\
0 & 1 & 0\\
0 & 0 & -\frac{\mathrm{deg}X}{2}\end{smallmatrix}\right)\left(\begin{smallmatrix}n_{0}\tilde{w}_{0}\\
n_{1}\tilde{w}_{1}\\
n_{2}\tilde{w}_{2}\end{smallmatrix}\right),\label{eq:PiX-PiZ-K3}\end{equation}
where $N_{x}$ and $N_{z}$ are unknown constants and $n_{k}:=\frac{1}{(2\pi i)^{k}}$.
These forms are expected in general to give an integral basis which
represents the mirror isomorphism $T_{\check{X}}\simeq(K(X),-\chi(*,**))$
with the bilinear form $\Sigma_{n}=\left(\begin{smallmatrix}0 & 0 & 1\\
0 & 2n & 0\\
1 & 0 & 0\end{smallmatrix}\right)$ $(\deg X=2n)$. The constants $N_{x,}N_{z}$ are determined by the
Griffiths tansversalities;\begin{equation}
\begin{matrix}\,^{t}\Pi\Sigma_{n}\Pi=\,^{t}\Pi\Sigma_{n}\frac{d\;}{dx}\Pi=0, & ^{t}\Pi\Sigma_{n}\frac{d^{2}\;}{dx^{2}}\Pi=\frac{-1}{(2\pi i)^{2}}C_{xx},\;\;\;\;\\
\,^{t}\tilde{\Pi}\Sigma_{n}\tilde{\Pi}=\,^{t}\tilde{\Pi}\Sigma_{n}\frac{d\;}{dz}\tilde{\Pi}=0,\, & ^{t}\tilde{\Pi}\Sigma_{n}\frac{d^{2}\;}{dz^{2}}\tilde{\Pi}=\frac{-1}{(2\pi i)^{2}}C_{xx}\left(\frac{dx}{dz}\right)^{2},\end{matrix}\qquad\qquad\label{eq:Griffiths-transversality}\end{equation}
where $C_{xx}=12/((1-34x+x^{2})x^{2})$ is the Griffiths-Yukawa coupling
\cite{CdOGP} normalized by $\deg X=12$. The following results are
parallel to those in \cite[Prop.2.10]{HoTa1}:
\begin{prop}
\label{pro:Monodromy-K3-deg12}{\rm (1)} The ansatz (\ref{eq:PiX-PiZ-K3})
with $N_{x}=N_{z}=1$ satisfies (\ref{eq:Griffiths-transversality}). 

\begin{myitem} \item[{\rm (2)}]  The two local solutions are related
under an analytic continuation along a path through the upper half
plane by $\Pi(x)=U_{xz}\tilde{\Pi}(z)$ with $U_{xz}=\left(\begin{smallmatrix}3 & 12 & -2\\
1 & 5 & -1\\
-2 & -12 & 3\end{smallmatrix}\right)$. 

\item[{\rm (3)}] Monodromy matrices $M_{c}$ of $\Pi(x)$ $(\tilde{M}_{c}$
of $\tilde{\Pi}(z))$ around each singular point $x=c$ of (\ref{eq:PF-K3})
are given by

\def\m{\text{{\bf -}}} 
$\qquad\qquad$
{\tablinesep=0.5pt \tabcolsep=2pt  
\begin{tabular}{|c|cccc|}  
\hline  
&{\small $x=0$} & {\small $a_1$} & {\small $a_2$} & {\small $\infty$} \\  
\hline 
{\small $M_c$} & 
$\left(\begin{smallmatrix}
1 & 0 & 0\\
1 & 1 & 0\\ 
\m6 & \m12 & 1
\end{smallmatrix}\right)$
& 
$\left(\begin{smallmatrix}
0 & 0 & 1\\ 
0 & 1 & 0\\ 
1 & 0 & 0\end{smallmatrix}\right)$ 
& 
$\left(\begin{smallmatrix} 
\m24 & 120 & 25\\ 
\m10 & 49 & 10\\ 
25 & \m120 & \m24
\end{smallmatrix}\right)$
& 
$\left(\begin{smallmatrix} 
49 & \m168 & \m24\\ 
21 & \m71 & -10\\ 
\m54 & 180 & 25
\end{smallmatrix}\right)$
\\
\hline
{\small ${\tilde M}_c$} & 
$\left(\begin{smallmatrix}
25 & 120 & \m24\\ 
\m15 & \m71 & 14\\ 
\m54 & \m252 & 49 
\end{smallmatrix}\right)$
& 
$\left(\begin{smallmatrix}
\m24 & \m120 & 25\\ 
10 & 49 & \m10\\ 
25 & 120 & \m24 
\end{smallmatrix}\right)$
& 
$\left(\begin{smallmatrix}
0 & 0 & 1\\
0 & 1 & 0\\
1 & 0 & 0
\end{smallmatrix}\right)$
& 
$\left(\begin{smallmatrix}
1 & 0 & 0\\ 
1 & 1 & 0\\ 
\m6 & \m12 & 1
\end{smallmatrix}\right)$
\\
\hline
\end{tabular}     }

\noindent and satisfy $M_{0}M_{a_{1}}M_{a_{2}}M_{\infty}=\id$ and
$\tilde{M}_{c}=U_{xz}^{-1}M_{c}U_{xz}$ with $U_{xz}^{-1}=\left(\begin{smallmatrix}3 & -12 & -2\\
-1 & 5 & 1\\
-2 & 12 & 3\end{smallmatrix}\right)$. 

\item[{\rm (4)}] $M_{c}$'s and $U_{xz}$ are given in terms of generators
of $\Gamma_{0}(6)_{+}$ in (\ref{eq:Gamma6}) by \[
M_{0}=R(T_{0}^{-1}),\; M_{a_{1}}=-R(S_{1}),\; M_{a_{2}}=-R(S_{2}S_{1}S_{2}),\; U_{xz}=R(S_{1}S_{2}).\]
In particular $M_{0,}M_{a_{1}},M_{a_{2}}\in O(\check{M}_{6}^{\perp})^{*}$
and $U_{xz}\in O(\check{M}_{6}^{\perp})\setminus O(\check{M}_{6}^{\perp})^{*}$
with the symmetric form $\Sigma_{6}$.\end{myitem}
\end{prop}
In Fig.~2.1, we see that the modular action of the element $S_{1}S_{2}\in\Gamma_{0}(6)_{+}\setminus\Gamma_{0}(6)_{+6}$
on $\mathbb{H}_{+}$ identifies the image of $D_{+}$ with that of
$D_{-}$ by exchanging the two cusp points.

\subsubsection{\label{sub:FM-functor-K3} {\itshape\bfseries FM functor $\Phi_{\sP}$
and Auteq$\, D^{b}(X)$} }

We can read more from the mirror isomorphism $T_{\check{X}}\simeq(K(X),-\chi(*,**))$
which comes from the monodromy calculations. Let us note that the
integral basis $\Pi(x)=\,^{t}(\Pi_{1},\Pi_{2},\Pi_{3})$ in Proposition
\ref{pro:Monodromy-K3-deg12} implicitly determines the corresponding
basis $(\gamma_{1},\gamma_{2},\gamma_{3})$ of the transcendental
lattice $T_{\check{X}}$. As for the basis of the lattice $(K(X),-\chi(*,**))$,
we may take \[
([\sE_{1}],[\sE_{2}],[\sE_{3}])=([\sO_{x}],[\sO_{h}]+6[\sO_{x}],-[\sI_{x}]),\]
with $0\to\sO_{X}(-h)\to\sO_{X}\to\sO_{h}\to0$, and $\sO_{x}$ the
skyscraper sheaf and $\sI_{x}$ the ideal sheaf of a point $x\in X$.
Note that we choose $[\sE_{2}]$ so that $ch([\sO_{h}]+6[\sO_{x}])=h$,
and hence we can verify $(-\chi([\sE_{i}],[\sE_{j}]))=\Sigma_{6}$
by Riemann-Roch theorem. Identifying these two basis, we have an explicit
isomorphism $T_{\check{X}}\simeq(K(X),-\chi(*,**))$ (this can be
done in general \cite[Sect.2.4]{HoIIa}). 

Actually, the identification of the two basis above is somehow canonical
from the viewpoint of homological mirror symmetry, since we can show
that the topology of $\gamma_{1}$ is isomorphic to the real two torus,
i.e.$\gamma_{1}\thickapprox T^{2}$. The identification of such torus
cycle with $\sO_{x}$ is justified from many aspects of the homological
mirror symmetry $D^{b}Fuk(\check{X})\simeq D^{b}(X)$ (see \cite{Ko,SYZ}).
Note also that $\gamma_{1}$ is isotropic in $T_{\check{X}}$ and
choosing such a vector in $T_{\check{X}}$ determines (almost uniquely,
i.e., up to signs) other bases with the specified intersection numbers
in the entries of $\Sigma_{6}$. Similar construction of the basis
of $\tilde{\Pi}(z)$ (or the cycles $\tilde{\gamma}_{1},\tilde{\gamma}_{2},\tilde{\gamma}_{3}$)
and the identification $\tilde{\gamma}_{1}\thickapprox T^{2}$ with
$\sO_{y}$ are valid for $(K(Y),-\chi(*,**))$. We denote by $h'$
the polarization of $Y$. 

Now recall that the Fourier-Mukai functor $\Phi_{\sP}:D^{b}(Y)\simeq D^{b}(X)$
is defined by the kernel $\sP$, the universal bundle over $X\times Y=X\times\mathcal{M}_{X}(2,h,3)$,
and hence we have $\Phi_{\sP}(\sO_{y})=\mathcal{P}_{y}$ with the
Mukai vector $ch(\mathcal{P}_{y})\sqrt{Todd_{X}}=2+h+3v$ $(v:=ch(\sO_{x}))$.
From this, we have\[
\begin{aligned}ch(\Phi_{\sP}(\sO_{y})) & =ch(\mathcal{\sP}_{y})=2+h+v=3v+h+2(1-v)\\
 & =3ch([\sE_{1}])+ch([\sE_{2}])-2ch([\sE_{3}]),\end{aligned}
\]
and identify this in the 1st column of the connection matrix $U_{xz}=R(S_{1}S_{2})$%
\footnote{The correspondence between the Chern characters $ch(\sP_{y})$=$ch(\Phi_{\sP}(\sO_{y}))$
for $\sP=\sP_{Y\to X}$ $(Y\in FM(X))$ and the elements in $\Gamma_{0}(n)_{+}\setminus\Gamma_{0}(n)_{+n}$
in general has been worked in \cite{Kaw}. %
}(note that we identify $\tilde{\gamma}_{1}$ with $\sO_{y}$). This
leads us to a conjecture that the continuation of the cycles $\tilde{\gamma}_{1},\tilde{\gamma}_{2},\tilde{\gamma}_{3}$
to $\gamma_{1},\gamma_{2},\gamma_{3}$ corresponds to the Fourier-Mukai
functor $\Phi_{\sP}:D^{b}(Y)\simeq D^{b}(X)$. Note that the analytic
continuation of $\Pi(x)$ connects cycles in the fibers around $x=0$
and those around $x=\infty$, but actually it comes from a Dehn twist
of $\check{X}$ because the local family around $x=0$ and $x=\infty$
are isomorphic as the family of $\check{M}_{6}$-polarizable K3 surfaces.
Dehn twists around $x=0,a_{1},a_{2},\infty$ are easy to be identified
from the standard forms of the monodromy matrices $M_{0},M_{a_{1}},\tilde{M}_{a_{2}}$
and $\tilde{M}_{\infty}$. They can be identified, respectively, with
the following Fourier-Mukai functors (see e.g. \cite{ST}):\[
(-)\otimes\sO_{X}(h),\quad\Phi_{\sI_{\Delta(X)}},\quad\Phi_{\sP}\circ\Phi_{\sI_{\Delta(Y)}}\circ\Phi_{\sP}^{-1}\;\text{ and \;}\Phi_{\sP}\circ\big((-)\otimes\sO_{Y}(h')\big)\circ\Phi_{\sP}^{-1},\]
where $\sI_{\Delta(X)}$ (resp. $\sI_{\Delta(Y)}$) is the ideal sheaf
of the diagonal $\Delta\subset X\times X$ (resp. $\Delta\subset Y\times Y$)
and $h'$ is the polarization of $Y$. From the above considerations,
and taking the monodromy relation into account, we naturally come
to a conjecture that the group $Auteq\, D^{b}(X)$ is generated by
the shift functor and the following Fourier-Mukai functors: \[
(-)\otimes\sO_{X}(h),\quad\Phi_{\sI_{\Delta(X)}}\text{\; and \;}\Phi_{\sP}\circ\Phi_{\sI_{\Delta(Y)}}\circ\Phi_{\sP}^{-1}.\]

\vskip0.5cm

\subsection{\label{sub:Some-other-aspects}Some other aspects }

From the example in the previous subsection, one may expect some relation
between the Fourier-Mukai numbers $|FM(X)|$ and the numbers of MUM
points in $\sD(\check{M}^{\perp})/O(\check{M}^{\perp})^{*}$. In fact,
S. Ma \cite{Ma} (see also \cite{Hart}) showed that the counting
formula in Theorem \ref{thm:HOLY-counting-F} allows such interpretation
if we identify MUM points with the standard cusps in the Baily-Satake
compactification of $\sD(\check{M}^{\perp})/O(\check{M}^{\perp})^{*}$.
From this viewpoint, we can read the counting formula as the number
of non-isomorphic decompositions of $\check{M}^{\perp}$ into $\check{M}^{\perp}=U\oplus M$
modulo the actions of $O(\check{M}^{\perp})^{*}$. Non-standard cusps
are 0-dimensional boundary points which correspond to the decompositions
$\check{M}^{\perp}=U(m)\oplus M$ $(m>1)$. In ref.~\cite{Ma}, the
counting formula has been generalized to incorporate non-standard
cusps, and it has been shown that the generalized formula counts the
number of twisted Fourier-Mukai partners, i.e., K3 surfaces $Y$ satisfying
$D^{b}(X)\simeq D^{b}(Y,\alpha)$ where $\alpha$ is an element of
the Brauer Group $Br(Y)$. See references \cite{HuySt,Cal} for the
derived categories of twisted sheaves on $Y$. 

\vskip3cm

\section{\textbf{\textup{Fourier-Mukai partners of Calabi-Yau threefolds I}} }

We define Calabi-Yau 3-folds by smooth, projective, three dimensional
varieties $X$ over $\mC$ which satisfy $K_{X}\simeq\sO_{X}$, $H^{1}(X,\sO_{X})=H^{2}(X,\sO_{X})=0$.
It is known, due to Bridgeland \cite{Br2}, that birational Calabi-Yau
3-folds $X,Y$ are derived equivalent, i.e., $D^{b}(X)\simeq D^{b}(Y)$.
Except this general theorem, however, not much is known about the
Fourier-Mukai partners of Calabi-Yau 3-folds. Here and in the next
section, we focus on two examples of pairs of Calabi-Yau 3-folds with
Picard number one which are Fourier-Mukai partners but not birational
to each other. In both cases, some similarity to the example of Mukai
in the last section will be observed in the fact that suitable projective
dualities play important roles in their constructions and also their
derived equivalences.

\vskip0.5cm

\subsection{\label{sub:Grassmannian-and-Pfaffian}Grassmannian and Pfaffian Calabi-Yau
threefolds }

The first example is Calabi-Yau 3-folds due to R{\o}dland. Let $\rG(2,7)$
be the Grassmannian of two dimensional subspaces in $\mC^{7}$. Consider
the Pl\"ucker embedding of $\rG(2,7)$ into $\mP(\wedge^{2}\mC^{7})$.
Then the projective dual of $\rG(2,7)$ is the Pfaffian variety $\mathrm{Pf}(4,7)$
in the dual projective space $\mP(\wedge^{2}(\mC^{*})^{7})$, i.e.,
the locus $\left\{ [c_{ij}]\in\mP(\wedge^{2}(\mC^{*})^{7})\mid\rank(c_{ij})\leq4\right\} .$
Let us consider general 7 dimensional linear subspace $L_{7}\subset\wedge^{2}(\mC^{*})^{7}$
and its orthogonal subspace $L_{7}^{\perp}\subset\wedge^{2}\mC^{7}$.
Then, similarly to the construction in Subsection \ref{sub:OG(5,10)-def},
we define \[
X=\rG(2,7)\cap\mP(L_{7}^{\perp})\subset\mP(\wedge^{2}\mC^{7}),\;\; Y=\mathrm{Pf(4,7)}\cap\mP(L_{7})\subset\mP(\wedge^{2}(\mC^{*})^{7}).\]
$X$ and $Y$, respectively, are called Grassmannian and Pfaffian
Calabi-Yau 3-folds.
\begin{prop}[R{\o}dland \cite{Ro}]
\label{pro:Rodland-X-Y} When $L_{7}$ is general, both $X$ and
$Y$ are smooth Calabi-Yau 3-folds with Picard number one and the
following invariants:\[
\begin{matrix}H_{X}^{3}=42, & c_{2}(X).H_{X}=84 & h^{1,1}(X)=1, & h^{2,1}(X)=50\\
H_{Y}^{3}=14, & c_{2}(Y).H_{Y}=56 & h^{1,1}(Y)=1, & h^{2,1}(Y)=50\end{matrix}\]
where $H_{X}$ and $H_{Y}$ are the ample generators of the Picard
groups, respectively.
\end{prop}
As for the smoothness, it is further known that $X$ is smooth if
and only if Y is smooth \cite{BC}. The equal Hodge numbers might
indicate a possibility that $X$ and $Y$ were birational to each
other \cite{Ba2}. However, looking the degrees $H_{X}^{3}=42$ and
$H_{Y}^{3}=14$ together with $\rho(X)=\rho(Y)=1$, we see that this
is not the case. 

In \cite{Ro}, R{\o}dland studied mirror symmetry of Pfaffian Calabi-Yau
threefold $Y$ and constructed a mirror family $\mathcal{Y}=\left\{ \check{Y}_{x}\right\} _{x\in\mP^{1}}$
by the so-called orbifold mirror construction. His construction starts
with a special family of Pfaffian Calabi-Yau 3-folds which admits
a Heisenberg group action \cite{GrPo}. By finding a suitable subgroup
of the Heisenberg group as the orbifold group, and making a crepant
resolutions for the singularities in the orbifold mirror construction,
the desired mirror Calabi-Yau 3-folds $\check{Y}$ with Hodge numbers
$h^{1,1}(\check{Y})=50,h^{2,1}(\check{Y})=1$ was obtained. Independently,
mirror symmetry of Grassmannian Calabi-Yau 3-folds $X$ was studied
in \cite{BCKvS} by the method of toric degeneration of Grassmannians.
It was recognized by these authors that the Picard-Fuchs differential
equations for these two families have exactly the same form but they
are distinguished by two different MUM points of the equation, as
we have witnessed in the equation (\ref{eq:PiX-PiZ-K3}). In particular,
it was observed that Gromov-Witten invariants ($g=0$) calculated
from the two MUM points ($x=0$ and $x=\infty$ in Subsection \ref{sub:Mirror-symmetry-Gr-Pf})
match to those for $X$ and Y, respectively. 

Later, in \cite{HoKo}, the calculation of Gromov-Witten invariants
$(g=0)$ have been extended to higher genus $(g\leq10)$ solving the
so-called BCOV holomorphic anomaly equation discovered in \cite[\hskip-3pt 2]{BCOV1}. 

\vskip0.5cm

\subsection{\label{sub:Derived-equivalence-Gr-Pf}Derived equivalence $D^{b}(X)\simeq D^{b}(Y)$}

As described in the previous subsection, there are similarities in
their constructions between the example of Fourier-Mukai partners
in Subsection \ref{sub:An-example-due-Mukai} and the Grassmannian
and Pfaffian Calabi-Yau 3-folds $X$ and Y. It is natural to expect
that $X$ and $Y$ are derived equivalent. In fact, the derived equivalence
is supported from the analysis of Gauged Linear Sigma Model (GLSM)
in physics \cite{HoriT}. The derived equivalence has been proved
mathematically in \cite{BC} and \cite{Ku2} (see also \cite{BDFIK,ADS}
for recent progresses). 

Let $\mathcal{Y}$ be the Pfaffian variety $\mathrm{Pf(4,7})$. $\mathcal{Y}$
is singular along $\mathcal{Y}_{sing}=\left\{ [c_{ij}]\mid\rk c\leq2\right\} $
and has a natural (Springer-type) resolution\begin{equation}
\tilde{\mathcal{Y}}=\left\{ ([c],[w])\mid w\subset\ker c\right\} \subset\mathcal{Y}\times\rG(3,7).\label{eq:resol-Pf-Y}\end{equation}
Since it is easy to see that all the fibers of the projection $\Lrho:\tilde{\mathcal{Y}}\to\rG(3,7)$
are isomorphic to $\mP^{5}$, $\tilde{\mathcal{Y}}$ is smooth. Let
us denote $\rG(2,7)$ by $\mathcal{X}$. Then we have $X=\mathcal{X}\cap\mP(L_{7}^{\perp})$
and also we can write $Y=\tilde{\mathcal{Y}}\cap\mP(L_{7})$ since
$\mathcal{Y}_{sing}$ is away from $\mP(L_{7})$ for general $L_{7}$.
Let us summarize our settings into the following diagram:\def\diagI{
\begin{xy}
(0,10)*++{\mathcal X}="A",
(30,10)*+{\tilde{\mathcal Y}}="C",
(30,0)*+{\mathcal Y}="D",
(0,3)*+{}="B1",
(18,3)*+{}="B2",
(8,0)*+{\rG(2,7)\;\;\;\;\rG(3,7)},
\ar@{=} "A";"B1",
\ar^{\pi} "C";"D",
\ar_{\rho} "C";"B2",
\end{xy}}\begin{equation}
\begin{matrix}{\diagI}\end{matrix}\label{eq:Gr-Pf-resl-diagram}\end{equation}
The proofs of the derived equivalence in \cite{BC} and \cite{Ku2}
uses a natural incidence correspondence between the two Grassmannians
in the diagram, which is given by \[
\Delta_{0}=\left\{ ([\xi],[w])\mid\dim(\xi\cap w)\geq1\right\} \subset\rG(2,7)\times\rG(3,7).\]
To sketch the proofs, let us consider the ideal sheaf $\sI_{\Delta_{0}}$
of $\Delta_{0}$ and define its pullback $\sI:=(\id\times\rho)^{*}\sI_{\Delta_{0}}$
on $\mathcal{X}\times\tilde{\mathcal{Y}}$. The restriction $I:=\sI\vert_{X\times Y}$
is an ideal sheaf on $X\times Y$. We regard $I$ as an object in
$D^{b}(X\times Y)$ and defines the Fourier-Mukai functor $\Phi_{I}(-):=R\pi_{X*}(L\pi_{Y}^{*}(-)\otimes I)$,
where $\pi_{X}$ and $\pi_{Y}$ are projections to $X$ and $Y$.
Then, Borisov and Caldararu proved the following
\begin{thm}[{\cite[Theorem 6.2]{BC}}]
\label{thm:DX-DY-Gr-Pf} $\Phi_{I}(-):D^{b}(Y)\to D^{b}(X)$ is an
equivalence. 
\end{thm}
\noindent The proof of the above theorem is based on the following
theorem for smooth projective varieties $X,Y$ and a Fourier-Mukai
functor $\Phi_{\sP}(-)=R\pi_{X*}(L\pi_{Y}^{*}(-)\otimes\sP)$ with
an object $\sP\in D^{b}(X\times Y)$ (see \cite[Thm.1.1]{BO}, \cite[Thm.1.1]{Br1},
\cite[Cor. 7.5, Prop. 7.6]{Huy}):
\begin{thm}
\label{thm:Ix-Iy-Thm}If $\sP$ a coherent sheaf on $X\times Y$ flat
over $Y$, then $\Phi_{\sP}:D^{b}(Y)\to D^{b}(X)$ is fully faithful
if and only if the following two conditions are satisfied:

\begin{myitem2}

\item[{\rm (i)}] For any point $x\in X$, it holds $\Hom(\sP_{x},\sP_{x})\simeq\mC$,
and

\item[{\rm (ii)}] if $x_{1}\not=x_{2}$, then $\Ext^{i}(\sP_{x_{1}},\sP_{x_{2}})=0$
for any i. 

\end{myitem2} Under these conditions, $\Phi_{\sP}$ is an equivalence
if and only of $\dim X=\dim Y$ and $\sP\otimes\pi_{X}^{*}\omega_{X}\simeq\sP\otimes\pi_{Y}^{*}\omega_{Y}$.
\end{thm}
It has been proved that the ideal sheaf $I$ is flat over $Y$, and
in fact, defines a flat family of curves parametrized by $Y$ \cite[Prop. 4.4]{BC}.
The condition $\Hom(I_{y},I_{y})\simeq\mC$ follows from a general
property of ideal sheaves of subschemes of dimension $\leq1$ in smooth
projective 3-folds {[}ibid,Prop.~4.5{]}. Hence, verifying the cohomology
vanishings \begin{equation}
\Ext^{\bullet}(I_{y_{1}},I_{y_{2}})=0\;\;(y_{1}\not=y_{2})\label{eq:ExtGrPf}\end{equation}
is the main part of the proof given in {[}ibid{]}. 

Kuznetsov formulates the derived equivalence as a consequence of the
homological projective duality (HPD) between $\rG(2,7)$ and $\mathrm{Pf}(4,7)$
(precisely, the non-commutative resolution of $\mathrm{Pf}(4,7)$).
In the proof given in \cite{Ku2}, the following locally free resolution
of the ideal sheaf $\sI$ on $\mathcal{X}\times\tilde{\mathcal{Y}}$
plays an important role: \begin{equation}
0\to\tS^{2}\sU\boxtimes\sO_{\tilde{\mathcal{Y}}}\to\sU\boxtimes\tilde{\sQ}\to\sO_{\mathcal{X}}\boxtimes\wedge^{2}\tilde{Q}\to\sI\otimes\sO_{\mathcal{X}\times\tilde{\mathcal{Y}}}(1,(1,0))\to0,\label{eq:GrPf-reol}\end{equation}
where $\sU$ is the universal bundle on $\rG(2,7)$, $\tilde{Q}$
is the universal quotient bundle on $\rG(3,7)$ and $\sO_{\mathcal{X}\times\tilde{\mathcal{Y}}}(1,(1,0)):=(\sO_{\mathcal{X}}(1)\boxtimes\rho^{*}\sO_{\rG(3,7)}(1))$
(see {[}ibid, Lemma 8.2{]}). The restriction of (\ref{eq:GrPf-reol})
to $\mathcal{X}\times\{y\}$ is nothing but the Eagon-Northcot complex
which was used for the proof of the vanishings (\ref{eq:ExtGrPf})
in \cite[Prop. 3.6]{BC}. Although we do not go into the details of
HPD, but for the comparison with the corresponding results in another
example in the next section it is useful to summarize some of the
main results in \cite{Ku2}. For that, let us introduce \textcolor{black}{the
following notation} for the sheaves that appear in (\ref{eq:GrPf-reol}):\begin{align*}
E_{3} & =\tS^{2}\mathcal{U},\; E_{2}=\mathcal{U},\; E_{1}=\sO_{\mathcal{X}};\;\; F_{3}=\sO_{\tilde{\mathcal{Y}}},\; F_{2}=\tilde{Q},\; F_{1}'=\wedge^{2}\tilde{Q},\end{align*}
and define the following full subcategories $\sA_{i}\subset D^{b}(\mathcal{X})$
$(i=0,...,6)$ and $B_{k}\subset D^{b}(\tilde{\mathcal{Y}})$ $(k=0,...,13)$:
\begin{equation}
\begin{aligned}\langle E_{3},E_{2},E_{1}\rangle=\sA_{0}=\sA_{1}=\cdots=\sA_{6}\subset D^{b}(\mathcal{X}),\\
\langle F_{1}^{*},F_{2}^{*},F_{3}^{*}\rangle=\sB_{0}=\sB_{1}=\cdots=\sB_{13}\subset D^{b}(\tilde{\mathcal{Y}}),\end{aligned}
\label{eq:Collections-A-B-Gr-Pf}\end{equation}
where we set $F_{1}:=F_{1}'/\sO_{\tilde{\mathcal{Y}}}(1,-1)$ with
$\sO_{\tilde{\mathcal{Y}}}(a,b)=\rho^{*}\sO_{\rG(3,7)}(a)\otimes\pi^{*}\sO_{\mathcal{Y}}(b)$. 
\begin{thm}[{\cite[Theorem 4.1]{Ku2}}]
\label{thm:HPD-Gr-Pf} Denote by $\sA_{i}(a)$, $\sB_{i}(a)$ the
twists of $\sA_{i}$, $\sB_{i}$ by $\sO_{\mathcal{X}}(a)$ and $\pi^{*}\sO_{\mathcal{Y}}(a)$,
respectively. Then 

{\rm (i)} $\langle\sA_{0},\sA_{1}(1),\cdots,\sA_{6}(6)\rangle$ is
a Lefschetz decomposition of $D^{b}(\mathcal{X})$, and 

{\rm (ii)} $\langle\sB_{13}(-13),\cdots,\sB_{1}(-1),\sB_{0}\rangle$
is a dual Lefschetz decomposition of $\tilde{D}^{b}(\mathcal{Y})$,

\noindent  where $\tilde{D}^{b}(\mathcal{Y})\subset D^{b}(\tilde{\mathcal{Y}}$)
is a full subcategory which is equivalent to $D^{b}(\sY,\sR$), the
bounded derived category of coherent sheaves of right $\sR$-modules
on $\mathcal{Y}$ with $\sR=\pi_{*}\sE nd(\sO_{\tilde{\mathcal{Y}}}\oplus\rho^{*}\tilde{\sU})$
and $\tilde{\mathcal{U}}$ the universal bundle on $\rG(3,7)$. 
\end{thm}
A (dual) Lefschetz decomposition is a special form of a semi-orthogonal
decomposition of a triangulated category \cite{BO}. In our case,
the vanishings \[
\Hom_{D^{b}(\tilde{\sY})}^{\bullet}(\sB_{i}(-i),\sB_{j}(-j))=0\;(i<j),\]
which are implied in (ii) of the above theorem, entail the desired
vanishings (\ref{eq:ExtGrPf}). 

\vskip0.8cm

\subsection{\label{sub:BPS-numbers-Gr-Pf}BPS numbers}

As noted in the previous subsection, the ideal sheaf $I_{y}\,(y\in Y$)
defines a family of curves on $X$. It can be shown by explicit calculations
with Macaulay2 that 
\begin{prop}
\label{pro:Curve-BPS-Gr-Pf}For a general point $y\in Y$, the ideal
sheaf $I_{y}$ defines a smooth curve on $X$ of genus $6$ and degree
14. 
\end{prop}
Expecting some relations to the moduli problems of ideal sheaves on
$X$, such as Donaldson-Thomas invariants of $X$ \cite{PT} or BPS
numbers \cite{HST}, it is interesting to seek a possibly related
number in the table of the BPS numbers calculated in \cite{HoKo}.
The relevant part of the table to the curves of Proposition \ref{pro:Curve-BPS-Gr-Pf}
reads as follows (with $d=14)$:\begin{equation}
\begin{tabular}{c|ccccccc}
 \mbox{\ensuremath{g}}  &  0  &  \mbox{\ensuremath{\cdots}}  &  6  &  7  &  8  &  9  &  10\\
\hline \mbox{\ensuremath{n_{g}^{X}(d)}}  &  2.67..\ensuremath{\times}1\ensuremath{0^{19}}  &  \ensuremath{\cdots} &  123676  &  392  &  7  &  0  &  0\end{tabular}\label{eq:Gr-Pf-BPS}\end{equation}
Unfortunately the BPS number $n_{6}^{X}(14)=123676$ is rather large
to find a relation to the curve defined by $I_{y}$. However, as noted
in \cite[(4-1.6)]{HoTa1}, we can observe that $n_{8}(14)=7$ counts
a well-known family of curves studied by Mukai, i.e., curves that
are linear sections of $\rG(2,6)$. Such curves appear in our setting
as \[
\rG(2,6)\cap\mP(L_{7}^{\perp})\subset\rG(2,7)\cap\mP(L_{7}^{\perp})=X,\]
and hence they are naturally parametrized by $\mP^{6}\simeq\left\{ \rG(2,6)\subset\rG(2,7)\right\} $.
General members of this family are smooth and of genus $8$ and degree
$14$. Then, following the counting {}``rule'' of BPS numbers \cite{GV},
we explain the number $n_{8}(14)=7$ as \[
n_{8}(14)=(-1)^{\dim\mP^{6}}e(\mP^{6})=7.\]
The counting {}``rule'' also tells us that such a generically smooth
family of curves of genus $g$ contributes to the numbers $n_{h}(d)\,(h\leq g$)
in a specified way {[}ibid{]}. Thus our observation above indicates
that there are contributions from at least two different families
of (generically) smooth curves in the BPS numbers $\left\{ n_{h}(14)\right\} _{h\leq8}$
in (\ref{eq:Gr-Pf-BPS}).

\vskip0.8cm

\subsection{\label{sub:Mirror-symmetry-Gr-Pf}Mirror symmetry}

Consider the mirror family $\check{\mathcal{Y}}=\left\{ \check{Y}_{x}\right\} _{x\in\mP^{1}}$
obtained from the orbifold mirror construction \cite{Ro}. The Picard-Fuchs
differential equation satisfied by the period integrals $w(x)=\int_{\gamma}\Omega(\check{Y}_{x})$
$(\gamma\in H_{3}(\check{Y}_{0},\mZ))$ has been determined by R{\o}dland
as $\sD_{x}w(x)=0$ with \[
\begin{matrix}\begin{aligned}\sD_{x}= & 9\,\theta_{x}^{4}-3\, x(15+102\,\theta_{x}+272\,\theta_{x}^{2}+340\,\theta_{x}^{3}+173\,\theta_{x}^{4})\\
 & -2\, x^{2}(1083+4773\,\theta_{x}+7597\,\theta_{x}^{2}+5032\,\theta_{x}^{3}+1129\,\theta_{x}^{4})\\
 & +2\, x^{2}(6+675\,\theta+2353\,\theta_{x}^{2}+2628\,\theta_{x}^{3}+843\,\theta_{x}^{4})\\
 & -x^{4}(26+174\,\theta_{x}+478\,\theta_{x}^{2}+608\,\theta_{x}^{3}+295\,\theta_{x}^{4})+x^{5}(\theta_{x}+1)^{4},\end{aligned}
\end{matrix}\]
and $\theta_{x}=x\frac{d\,}{dx}$. As described in Subsection \ref{sub:Grassmannian-and-Pfaffian},
the operator $\sD_{x}$ is the same as that of X in \cite{BCKvS,ES}
and Gromov-Witten invariants of $X$ and $Y$ are calculated, respectively,
from the MUM points at $x=0$ and $z=\frac{1}{x}=0$. Although the
geometry of the family is rather complicated (cf. Subsection \ref{sub:Mirror-symmetry-Reye}),
monodromy calculations proceeds in a similar way to Subsection \ref{sub:An-example-due-Mukai}.
The Riemann's $\sP$-scheme is\[    \left\{ {\tablinesep=-5pt \tabcolsep=2pt   
\begin{tabular}{cccccc}  
\small {\small $\;0\;$} & {\small $\;\alpha_1\;$} & {\small $\;\alpha_2\;$} 
& {\small $\;\alpha_3\;$}  & {\small $\;3\;$} & {\small $\;\infty\;$} \\  
\hline  0 &    0    &   0    &   0    &  0  &  1      \\  
0 &    1    &   1    &   1    &  1  &  1      \\  
0 &    1    &   1    &   1    &  3  &  1      \\  
0 &    2    &   2    &   2    &  4  &  1      \\ \end{tabular} }    \right\},  \]where $\alpha_{k}$ are the (real) roots of the 'discriminant' $1-57x-289x^{2}+x^{3}=0$
and $x=3$ is an apparent singularity with no monodromy (with order
$\alpha_{2}<0<\alpha_{1}<3<\alpha_{3})$. The symplectic and integral
basis of the solution can be obtained by making ansatz similar to
those in Subsection \ref{sub:Monodromy-calc-K3} (see also \cite{DM,ES}).
In fact, its full details are completely parallel to \cite[(2-5.1)-(2-5.7)]{HoTa1}
assuming two local solutions of the forms, \[
\Pi(x)=N_{x}\left(\begin{smallmatrix}1 & 0 & 0 & 0\\
0 & 1 & 0 & 0\\
\beta & a & \kappa/2 & 0\\
\gamma & \beta & 0 & -\kappa/6\end{smallmatrix}\right)\left(\begin{smallmatrix}n_{0}w_{0}(x)\\
n_{1}w_{1}(x)\\
n_{2}w_{2}(x)\\
n_{3}w_{3}(x)\end{smallmatrix}\right),\;\tilde{\Pi}(z)=N_{z}\left(\begin{smallmatrix}1 & 0 & 0 & 0\\
0 & 1 & 0 & 0\\
\tilde{\beta} & \tilde{a} & \tilde{\kappa}/2 & 0\\
\tilde{\gamma} & \tilde{\beta} & 0 & -\tilde{\kappa}/6\end{smallmatrix}\right)\left(\begin{smallmatrix}n_{0}\tilde{w}{}_{0}(z)\\
n_{1}\tilde{w}_{1}(z)\\
n_{2}\tilde{w}_{2}(z)\\
n_{3}\tilde{w}_{3}(z)\end{smallmatrix}\right).\]
Here we summarize only the results of the monodromy matrices.
\begin{prop}
{\rm (1)} When $N_{x}=N_{z}=1,a=\tilde{a}=0$ and \[
(\kappa,\beta,\gamma)=\begin{matrix}\left(H_{X}^{3},-\frac{c_{2}.H_{X}}{24},-\frac{\zeta(3)e(X)}{(2\pi i)^{3}}\right)\end{matrix},\;\tilde{(\kappa},\tilde{\beta},\tilde{\gamma)}=\begin{matrix}\left(H_{Y}^{3},-\frac{c_{2}.H_{Y}}{24},-\frac{\zeta(3)e(Y)}{(2\pi i)^{3}}\right)\end{matrix},\]
the solutions $\Pi(x)$ and $\tilde{\Pi}(z)$ are integral and symplectic
with respect to the symplectic form $S=\left(\begin{smallmatrix}0 & 0 & 0 & 1\\
0 & 0 & 1 & 0\\
0 & -1 & 0 & 0\\
-1 & 0 & 0 & 0\end{smallmatrix}\right)$. These are analytically continued along a path in the upper-half
plane as $\Pi_{x}(x)=U_{xz}\tilde{\Pi}(z)$ by a symplectic matrix
$U_{xz}=\left(\begin{smallmatrix}-3 & 7 & -1 & 4\\
0 & 3 & 0 & 1\\
14 & 0 & 5 & -7\\
0 & -14 & 0 & -5\end{smallmatrix}\right)$ with its inverse $U_{xz}^{-1}=\left(\begin{smallmatrix}-5 & -7 & -1 & -4\\
0 & 5 & 0 & 1\\
14 & 0 & 3 & 7\\
0 & -14 & 0 & -3\end{smallmatrix}\right)$.

\begin{myitem}

\item[{\rm (2)}] The monodromy matrices $M_{c}$ of $\Pi(x)$ $(\tilde{M}_{c}$
of $\tilde{\Pi}(z))$ around each singular point $c$ are symplectic
with respect to $S$, and they are given by $($with $\tilde{M}_{c}=U_{xz}^{-1}M_{c}U_{xz})$

\def\m{\text{{\bf -}}} \hspace{-13pt}
{\tablinesep=0.5pt \tabcolsep=-0.3pt  
\begin{tabular}{|c|ccccc|}  
\hline  
&{\small $x=0$} & {\small $\alpha_1$} & {\small $\alpha_2$} 
& {\small $\alpha_3$} & {\small $\infty$} \\  
\hline 
{\small $M_c$} & 
$\left(\begin{smallmatrix}
  1& 0 & 0 & 0   \\ 
  1& 1 & 0 & 0   \\
  21& 42 & 1 & 0   \\
  \m14& \m21 & \m1 & 1   \\ 
\end{smallmatrix}\right)$
& 
$\left(\begin{smallmatrix}
 1 & 0 & 0 & 1   \\ 
 0 & 1 & 0 & 0   \\
 0 & 0 & 1 & 0   \\
 0 & 0 & 0 & 1   \\ 
\end{smallmatrix}\right)$ 
& 
$\left(\begin{smallmatrix} 
 15& \m14 & 2 & 4   \\ 
 7 & \m6 & 1 & 2   \\
 49 & \m49 & 8 & 14   \\
 \m49 & 49 & \m7 & \m13   \\ 
\end{smallmatrix}\right)$
& 
$\left(\begin{smallmatrix} 
 1 & 42 & 0 & 9   \\ 
 0 & 1 & 0 & 0   \\
 0 & \m196 & 1 & \m42   \\
 0 & 0 & 0 & 1   \\ 
\end{smallmatrix}\right)$
& 
$\left(\begin{smallmatrix} 
 85 & \m14 & 16 & 42   \\ 
 6 & \m6 & 1 & 2   \\
 \m322 & 7 & \m62 & \m168   \\
 \m35 & 28 & \m6 & \m13   \\ 
\end{smallmatrix}\right)$
\\
\hline
{\small ${\tilde M}_c$} & 
$\left(\begin{smallmatrix}
 \m27 & 322 & \m8 & 126   \\ 
 13 & \m125 & 4 & \m50   \\
 7 & \m308 & 1 & \m112   \\
 \m42 & 385 & \m13 & 155   \\ 
\end{smallmatrix}\right)$
& 
$\left(\begin{smallmatrix}
 1 & 70 & 0 & 25   \\ 
 0 & 1 & 0 & 0   \\
 0 & \m196 & 1 & \m70   \\
 0 & 0 & 0 & 1   \\ 
\end{smallmatrix}\right)$
& 
$\left(\begin{smallmatrix}
 \m27 & 0 & \m8 & 16   \\ 
 14 & 1 & 4 & \m8   \\
 0 & 0 & 1 & 0   \\
 \m49 & 0 & \m14 & 29   \\ 
\end{smallmatrix}\right)$
& 
$\left(\begin{smallmatrix}
 1 & 0 & 0 & 1   \\ 
 0 & 1 & 0 & 0   \\
 0 & 0 & 1 & 0   \\
 0 & 0 & 0 & 1   \\ 
\end{smallmatrix}\right)$
& 
$\left(\begin{smallmatrix}
 1 & 0 & 0 & 0   \\ 
 1 & 1 & 0 & 0   \\
 7 & 14 & 1 & 0   \\
 \m7 & \m7 & \m1 & 1   \\ 
\end{smallmatrix}\right)$
\\
\hline
\end{tabular}     }

\noindent and satisfy $M_{a_{2}}M_{0}M_{a_{1}}M_{a_{3}}M_{\infty}=\id$. 

\end{myitem}
\end{prop}
As before, the integral basis $\Pi(x)$=$\left(\Pi_{1},\Pi_{2},\Pi_{3},\Pi_{4}\right)$
implicitly determines the corresponding integral cycles $\gamma_{i}$,
likewise for $\tilde{\Pi}(z)$ with the corresponding integral cycles
$\tilde{\gamma_{i}}(i=1,..,4)$. From the geometry of the family,
one can see that $\gamma_{1}\thickapprox\tilde{\gamma}_{1}\thickapprox T^{3}$
and also $\ensuremath{\gamma_{4}\thickapprox\tilde{\gamma}_{4}\thickapprox S^{3}}$
about the topologies of the cycles. Form the homological mirror symmetry,
these cycles may be identified with the skyscraper sheaves $\sO_{x}$,$\sO_{y}$($x\in X,y\in Y$)
and the structure sheaves $\sO_{X}$,$\sO_{Y}$ as was the case in
Subsection \ref{sub:FM-functor-K3}. Unfortunately we do not see directly
the relation $ch(\Phi_{I}(\sO_{y}))=ch(I_{y})$ in the 1st column
of $U_{xz}$ as before. However, we believe that if we take suitable
auto-equivalences into account, in other words, if we change the path
of the analytic continuation, we can identify the Chern character
in the connection matrix. Recently, precise analysis of the co-called
hemi-sphere partition functions of GLSMs \cite{HoriR} have been developed.
The analysis provides a concrete recipe to connect the cycles to the
objects in derived category (of matrix factorizations), and also reproduces
the connection matrix of the analytic continuation \cite{EHKR}. We
expect that the new method provides us new insights into more details
of the above problem. Also, the significant progresses made in refs
\cite{Ha,BDFIK,DS} in the mathematical aspects of GLSMs are expected
to provide us powerful tools to look into the derived categories of
Fourier-Mukai partners and also their mirror symmetry.

\vskip3cm

\section{\textbf{\textup{Fourier-Mukai partners of Calabi-Yau threefolds II }}}

Here we continue our exposition by the second example which was found
recently by the present authors \cite[\hskip-3pt 2,3,4]{HoTa1}.

\vskip0.5cm

\subsection{Reye congruences Calabi-Yau 3-folds and double coverings}

In \cite{HoTa1}, we have found that R{\o}dland's construction of
a pair of Calabi-Yau 3-folds has a natural counterpart in the projective
space of symmetric matrices $\mP(\tS^{2}\mC^{5})$. Hereafter, we
will fix $V=\mC^{5}$ and denote by $V_{k}$ a $k$-dimensional subspace
of $V$. 

We have found in {[}ibid{]} that the tower of secant varieties of
$v_{2}(\mP(V))$ in $\mP(\tS^{2}V)$ and the corresponding (reversed)
tower in $\mP(\tS^{2}V^{*})$ entail a similar duality of Calabi-Yau
3-folds. For the construction, we start with $\tS^{2}\mP(V)$, i.e.,
the symmetric product of $\mP(V)$ as the counterpart of the Grassmannian
$\rG(2,7)\subset\mP(\wedge^{2}\mC^{7})$. $\tS^{2}\mP(V)$ is the
first secant variety of $v_{2}(\mP(V))$ and can be considered as
the rank 2 locus of symmetric matrices $[c_{ij}]\in\mP(\tS^{2}V)$.
It is singular along the $v_{2}(\mP(V))$, i.e., the rank 1 locus.
The precise definition of the Pfaffian counterpart will be introduced
in the next section, but here we only describe the resulting Calabi-Yau
3-fold starting with the rank 4 locus in the dual projective space
$\mP(\tS^{2}V^{*})$,\[
\Hes:=\left\{ [a_{ij}]\in\mP(\tS^{2}V^{*})\mid\det(a_{ij})=0\right\} .\]
$\Hes$ is singular along the locus $\Hes_{3}$ with $\Hes_{k}:=\left\{ \rk(a_{ij})\leq k\right\} $.
As before, we consider a general five dimensional linear subspace
$L_{5}\subset\tS^{2}V^{*}$ and its orthogonal linear subspace $L_{5}^{\perp}\subset\tS^{2}V$.
Then we define\[
X=\tS^{2}\mP(V)\cap\mP(L_{5}^{\perp})\subset\mP(\tS^{2}V),\; H=\Hes\cap\mP(L_{5})\subset\mP(\tS^{2}V^{*}).\]

\begin{prop}[Hosono-Takagi \cite{HoTa1}]
\label{pro:HT-invariant-XY}{\rm (1)} When $L_{5}$ is general,
$X$ is a smooth Calabi-Yau 3-fold with $Pic(X)\simeq\mZ\oplus\mZ_{2}$
and the following invariants:\[
H_{X}^{3}=35,c_{2}.H_{X}=50,h^{1,1}(X)=1,h^{2,1}(X)=51,\]
\;\;where $H_{X}$ is the generator of the free part of $Pic(X)$. 

\begin{myitem}

\item[{\rm (2)}] When $L_{5}$ is general, H is a determinantal quintic
hypersurface in $\mP(L_{5})\simeq\mP^{5}$, which is singular along
a smooth curve $C_{H}$ of genus 26 and degree 20 with $A_{1}$ type
singularities. \item[{\rm (3)}] There is a double covering $Y\to H$
branched along $C_{H}$. Furthermore, $Y$ is a smooth Calabi-Yau
3-fold with $Pic(Y)=\mZ H_{Y}$ and \[
H_{Y}^{3}=10,c_{2}.H_{Y}=40,h^{1,1}(Y)=1,h^{2,1}(Y)=51.\]
\end{myitem} 
\end{prop}
If we do parallel constructions with $V=\mC^{4}$, we obtain an Enriques
surface for $X$. From historical reasons, this Enriques surface $X$
is called Reye congruence, or more precisely, Cayley model of Reye
congruence (see \cite{Cossec}). In our case of $V=\mC^{5}$, Reye
congruence $X$ is a Calabi-Yau 3-fold and is paired with another
Calabi-Yau 3-fold $Y$ as above. It is easy to see that $Y$ is not
birational to $X$ by the same arguments as described below Proposition
\ref{pro:Rodland-X-Y}. In addition to this, we can show \cite{HoTa4}
the derived equivalence $D^{b}(X)\simeq D^{b}(Y)$, which will be
sketched in the next subsection. Here it should be worth while noting
the following interesting properties of $X$ and $Y$ (\cite[Prop. 3.5.3, 4.3.4]{HoTa3},
\cite[Prop.3.2.1]{HoTa4}): 
\begin{prop}
{\rm (1)} $\pi_{1}(X)\simeq\mZ_{2}$. {\rm (2)} $\pi_{1}(Y)\simeq0$
and the Brauer group of $Y$ contains a non-trivial 2-torsion element.
\end{prop}
As argued in {[}ibid.,Sect.9.2{]}, one can show an exact sequence,
\[
0\to\mZ_{2}\to\mathrm{Br}(Y)\to\mathrm{Br}(X)\to0.\]
If $\mathrm{Br}(Y)\simeq\mZ_{2}$, then $\mathrm{Br}(X)\simeq0$ and
this indicates the invariance of the product of (abelianization of)
$\pi_{1}$ and the Brauer group, but not each factor, under the derived
equivalence (see \cite{Ad,S} for details).

\vskip0.5cm

\subsection{Derived equivalence $D^{b}(X)\simeq D^{b}(Y)$}

Here we sketch our proof of the derived equivalence. As we saw in
the preceding subsection, our construction of the pair $(X,Y)$ is
parallel to R{\o}dland's construction of Grassmannian-Pfaffian Calabi-Yau
manifolds. We can pursue this parallelism toward the proof of the
derived equivalence, although the projective geometries become more
involved, and we have only partial results about the HPD (corresponding
to Theorem \ref{thm:HPD-Gr-Pf}) in our case.

\subsubsection{{\itshape\bfseries Resolutions}}

Let $\chow:=\tS^{2}\mP(V)$. $X$ is defined by a linear section of
$\chow$ as $X=\chow\cap\mP(L_{5}^{\perp})$. We see that $\chow$
plays a similar role of $\rG(2,7)$ in R{\o}dland's example, however
there is a difference in that $\chow$ is singular along the Veronese
embedding of $\mP(V)$, $v_{2}(\mP(V))\subset\chow\subset\mP(\tS^{2}V)$.
For this singularity, we have the following natural resolution,\def\Hilb{
\begin{xy} 
(7,0)*+{\hchow:=\text{Hilb}^2\mP(V)},
(-3,-2)*+{\;}="A1",
(3,-2)*+{\;}="A2",
(-15,-12)*+{\chow}="B",
(15,-12)*+{\rG(2,V),}="C",
\ar_f "A1";"B",
\ar^g "A2";"C",
\end{xy} 
}\[
\begin{matrix}\Hilb\end{matrix}\]
where $\mathrm{Hilb}^{2}\mP(V)$ is the Hilbert scheme of two points
on $\mP(V)$ and $f$ is the Hilbert-Chow morphism. The morphism $g$
sends points $x\in\hchow$ to the points $g(x)\in\rG(2,V)$ representing
the lines determined by $x$. The fiber over $[V_{2}]\in\rG(2,V)$
is $g^{-1}([V_{2}])\simeq\tS^{2}\mP(V_{2})\simeq\mP^{2}$. By our
genericity assumption of $L_{5}$, $X=\chow\cap\mP(L_{5}^{\perp})$
is smooth (see Proposition \ref{pro:HT-invariant-XY}) and hence $\mP(L_{5}^{\perp})$
is away from the singularity of $\hchow$, therefore we may consider
our linear intersection in $\hchow$, i.e., $X=\hchow\cap\mP(L_{5}^{\perp})$.
Again, by the same reasoning, we have $g(X)\simeq X$, i.e., we have
isomorphic image $g(X)$ of $X$ in $\rG(2,V)$. Historically, the
image $g(X)\subset\rG(2,V)$ is called a Reye congruence. 

$\Hes$ is singular along the rank $\leq3$ locus $\Hes_{3}$. Expecting
a (partial) resolution of the singularity, we consider the following
(Springer-type) pairing of singular quadrics and planes therein (cf.
(\ref{eq:resol-Pf-Y})): \[
\Zpq:=\left\{ ([Q],[\Pi])\mid\mP(\Pi)\subset Q\right\} \subset\Hes\times\rG(3,V),\]
where $[Q]\in\Hes$ represents the point corresponding to a singular
quadric $Q$. It is easy to see that all the fibers of the projection
$\Zpq\to\rG(3,V)$ are isomorphic to $\mP^{8}$ since they consist
of quadrics that contain a fixed plane $\mP(\Pi)\subset\mP(V)$. Hence,
we see that $\Zpq$ is smooth. However we have $\dim\Zpq=6+8=14$,
while $\dim\Hes=\dim\mP(\tS^{2}V)-1=13$, and hence $\Zpq\to\Hes$
can not be a resolution of $\Hes$ that we expect. To remedy the situation,
we consider the Stein factorization $\hcoY$ of the morphism $\Zpq\to\Hes$
as follows:\def\SteinY{
\begin{xy}
(0,0)*+{\Zpq}="A",
(0,-10)*+{\hcoY}="B",
(0,-20)*+{\Hes}="C",
(12,-20)*+{\subset \mP(\tS^2V^*),},
(25,0)*+{\rG(3,V)}="D",
\ar^{\mP^8\text{-bundle}\quad} "A";"D",
\ar^{\pi_\Zpq \;}_{ \text{connected fibers} } "A";"B",
\ar^{\rho_\hcoY \;}_{ 2:1} "B";"C",
\end{xy}}\begin{equation}
\begin{matrix}\SteinY\end{matrix}\label{eq:def-hcoY-Stein-fact}\end{equation}
where $\pi_{\Zpq}:\Zpq\to\hcoY$ has connected fibers and $\rho_{\hcoY}$
is a finite morphism by definition. From the above dimension counting,
the connected fibers generically have dimension $\dim\Zpq-\dim\Hes=1$.
As for the finite morphism $\rho_{\hcoY}$, looking into the families
of planes in a singular quadric, it is easy to see that $\rho_{\hcoY}$
is generically $2:1$ and has its ramification along the singular
locus $\mathrm{Sing}(\Hes)=\Hes_{3}$. This corresponds to the covering
we observed in (3) of Proposition \ref{pro:HT-invariant-XY}. In fact,
about the singular locus of $\hcoY$, we can see $\text{Sing}(\hcoY)=\Hes_{2}$
\cite[Prop.5.7.2]{HoTa3} where we identify the inverse image $\rho_{\hcoY}^{-1}(\Hes_{3})$
in $\hcoY$ with $\Hes_{3}$. Hence the covering $\hcoY$ changes
the singular locus of $\Hes$ to a smaller one. If the linear subspace
$L_{5}$ is general, then since $\mP(L_{5})\cap\Hes_{2}=\emptyset$,
the singularities in the linear section $H=\Hes\cap\mP(L_{5})$ is
removed by $\rho_{\hcoY}$. This is exactly the smooth double covering
$Y$ in (3) of Proposition \ref{pro:HT-invariant-XY}. We write the
double cover of $H$ simply by $Y=\hcoY\cap\mP(L_{5})$ with understanding
the pullback of $\mP(L_{5})$ to $\hcoY$. A natural resolution $\widetilde{\hcoY}\to\hcoY$
follows by studying geometries of singular quadrics $\Hes$ \cite{HoTa3},
which is interesting by itself from the projective geometry of quadrics
\cite{Tyurin}. Birational geometry of $\hcoY$ and $\widetilde{\hcoY}$
will be described in Section \ref{sec:Birational-Geometry-of-hcoY}
by introducing other birational models of $\hcoY$. 

It would be helpful now to write our $X$ and $Y$ in terms of the
resolutions $\hchow$ and $\widetilde{\hcoY}$ as \[
X=\hchow\cap\mP(L_{5}^{\perp}),\;\;\; Y=\widetilde{\hcoY}\cap\mP(L_{5}).\]
The derived equivalence follows from certain ideal sheaf on $\widetilde{\hcoY}\times\hchow$
constructed in a parallel way to the Grassmannian-Pfaffian Calabi-Yau
3-folds. The following proposition is a part of the birational geometry
of $\hcoY$ (see Fig.~5.2):
\begin{prop}
\label{pro:G2-Z2-to-Y}\begin{myitem}  \item{\rm (1)} There exists
a resolution $\rho_{\widetilde{\hcoY}}:\widetilde{\hcoY}\to\hcoY$. 

\item{\rm (2)} There exists a blow-up $\hcoY_{2}\to\widetilde{\hcoY}$,
and over $\hcoY_{2}$ there is a generically conic bundle $\pi_{2'}:\Zpq_{2}\to\hcoY_{2}$
that admits a morphism $\mu_{2}:\Zpq_{2}\to\rG(3,V)$.

\end{myitem}
\end{prop}
We summarize the resolutions and morphisms as follows (cf. (\ref{eq:Gr-Pf-resl-diagram})
):\def\YZX{ 
\begin{xy}  
(0,0)*+{\Zpq_2}="A", 
(0,-13)*+{\hcoY_2}="B", 
(-15,-13)*+{\widetilde{\hcoY}}="C", 
(-30,-13)*+{\hcoY}="D", 
(22,-13)*+{\rG(3,V)\quad\rG(2,V)}, 
(11,-10)*+{ }="p1",
(32,-10)*+{ }="p2",
(40,0)*+{\hchow}="X", 
\ar_{\pi_{2'}} "A";"B", 
\ar_{\tilde{\rho}_2} "B";"C",
\ar_{\rho_{\widetilde{\hcoY}}} "C";"D",
\ar^{\mu_2} "A";"p1",
\ar_{g} "X";"p2",
\end{xy} }\begin{equation}
\begin{matrix}\YZX\end{matrix}\label{eq:diag-YZX}\end{equation}

\subsubsection{{\itshape\bfseries Incidence relation $\Delta_{0}$}}

In the diagram (\ref{eq:diag-YZX}), we introduce the following incidence
relation $\Delta_{0}$: \[
\Delta_{0}=\left\{ ([V_{3}],[V_{2}])\mid V_{3}\supset V_{2}\right\} \subset\rG(3,V)\times\rG(2,V),\]
and consider its ideal sheaf $\sI_{\Delta_{0}}$. Pulling this back
to $\Zpq_{2}\times\hchow$, we obtain $\sI_{\Delta_{2}}=(\mu_{2}\times g)^{*}\sI_{\Delta_{0}}$.
Since the variety $\Delta_{0}$ is nothing but the flag variety $F(2,3,V)$,
we have locally free resolution,\begin{equation}
0\to\wedge^{4}(\eQ^{*}\boxtimes\eF)\to\wedge^{3}(\eQ^{*}\boxtimes\eF)\to\wedge^{2}(\eQ^{*}\boxtimes\eF)\to\eQ^{*}\boxtimes\eF\to\sI_{\Delta_{2}}\to0,\label{eq:resol-D0}\end{equation}
where \begin{align*}
0 & \to\eS\to V\otimes\sO_{\rG(3,V)}\to\eQ\to0 & \text{and} &  & 0\to\eF\to V\otimes\sO_{\rG(2,V)}\to\eG\to0\end{align*}
are the universal sequences on the Grassmannians $\rG(3,V)$ and $\rG(2,V)$
($\rk\eS=3,\rk\eF=2$), respectively. Roughly speaking, the direct
image $(\tilde{\rho}_{2}\times\id)_{*}\circ(\pi_{2'}\times\id)_{*}\sI_{\Delta_{2}}$
is the ideal sheaf $\sI$ on $\widetilde{\hcoY}\times\hchow$ which
corresponds to the one used in the Grassmannian-Pfaffian case in \cite{BC}
and \cite{Ku2}. In actual calculation of the direct image, however,
we need to use the structure of the conic bundle. Hence we first restrict
the generically conic bundle to a conic bundle $\pi_{2'}^{o}:\Zpq_{2}^{o}\to\hcoY_{2}^{o}:=\hcoY_{2}\setminus\sP_{\sigma}$,
where $\sP_{\sigma}$ is a certain subvariety of dimension $7$, and
define $\sI^{o}:=(\tilde{\rho}_{2}^{o}\times\id)_{*}\circ(\pi_{2'}^{o}\times\id)_{*}\sI_{\Delta_{2}}$
with the corresponding restriction $\tilde{\rho}_{2}^{o}:\hcoY_{2}^{o}\to\widetilde{\hcoY}^{o}$.
Then $\sI=\iota_{*}\sI^{o}$ under the inclusion $\iota:\widetilde{\hcoY}^{o}\hookrightarrow\widetilde{\hcoY}$
is the precise definition of the ideal sheaf $\sI$.

\subsubsection{{\itshape\bfseries Derived equivalence}}

The proof of derived equivalence in \cite{HoTa4} proceeds by constructing
the Fourier-Mukai functor with the kernel $I=\sI\vert_{Y\times X}$
as in Subsection \ref{sub:Derived-equivalence-Gr-Pf}. In the paper
{[}ibid{]}, we have obtained a locally free resolution of the ideal
sheaf $\sI$ starting with (\ref{eq:resol-D0}). To describe the results,
we introduce locally free sheaves on $\widetilde{\hcoY}$.
\begin{prop}
There exists locally free sheaves $\tilde{\sS}_{L},\tilde{\sT},\tilde{\sQ}$
on $\widetilde{\hcoY}$ which satisfy\[
\begin{matrix}\begin{aligned}\pi_{2'*}\left\{ \mu_{2}^{*}\sO_{\rG(3,V)}(1)\right\} \simeq\tilde{\rho}_{2}^{*}\tilde{\sS}_{L}^{*},\quad\pi_{2'*}(\mu_{2}^{*}\eQ)\simeq\tilde{\rho}_{2}^{*}\tilde{\sT},\qquad\\
\;\;\pi_{2'*}\left\{ \big(\mu_{2}^{*}\tS^{2}\eW\big)\otimes\mu_{2}^{*}\sO_{\rG(3,V)}(-1)\right\} \simeq\tilde{\rho}_{2}^{*}\big(\tilde{\sQ}\otimes\sO_{\widetilde{\hcoY}}(-M_{\widetilde{\hcoY}})\big),\end{aligned}
\end{matrix}\]
where $M_{\widetilde{\hcoY}}$ is the divisor corresponding to $\rho_{\widetilde{\hcoY}}^{*}\circ\rho_{\hcoY}^{*}\sO_{\Hes}(1)$. \end{prop}
\begin{proof}
See \cite[Prop.5.6.4]{HoTa4} and \cite[Prop.6.1.2,6.2.3]{HoTa3}.
\end{proof}
We denote by $L_{\hchow}$ (resp. $H_{\hchow}$) the divisor on $\hchow$
corresponding to $g^{*}\sO_{\rG(2,V)}(1)$ (resp. $g^{*}\sO_{\chow}(1)$).
Then, we have 
\begin{thm}[{\cite[Theorem 5.1.3]{HoTa4}}]
We have the following locally free resolution: \[
\begin{matrix}\begin{aligned}0\to\tilde{\sS}_{L}\boxtimes\sO_{\hchow}\to\tilde{\sT}^{*}\boxtimes g^{*}\eF^{*}\to\big(\sO_{\widetilde{\hcoY}}\boxtimes g^{*}\tS^{2}\eF^{*}\big)\oplus\big(\tilde{\sQ}^{*}(M_{\widetilde{\hcoY}})\boxtimes\sO_{\hchow}(L_{\hchow})\big)\\
\to\sI\otimes\big(\sO_{\widetilde{\hcoY}}(M_{\widetilde{\hcoY}})\boxtimes\sO_{\hchow}(2L_{\hchow})\big)\to0.\end{aligned}
\end{matrix}\]

\end{thm}
Extracting each term of the above resolution of $\sI$, we define
the following notation:\[
\begin{matrix}\begin{aligned}(\sE_{3},\sE_{2},\sE_{1a},\sE_{1b})= & (\tilde{S}_{L},\tilde{\sT}^{*},\sO_{\widetilde{\hcoY}},\tilde{Q}^{*}(M_{\widetilde{\hcoY}})),\\
(\sF_{3},\sF_{2},\sF_{1a}',\sF_{1b})= & (\sO_{\hchow},g^{*}\eF^{*},g^{*}\tS^{2}\eF^{*},\sO_{\hchow}(L_{\hchow})),\end{aligned}
\end{matrix}\]
and set $\sF_{1a}=\sF_{1a}'/\sO_{\hchow}(-H_{\hchow}+2L_{\hchow})$.
Now corresponding to (\ref{eq:Collections-A-B-Gr-Pf}) in Subsection
\ref{sub:Derived-equivalence-Gr-Pf}, we define the following full-subcategories\[
\begin{matrix}\begin{aligned}\langle\sE_{3},\sE_{2},\sE_{1a},\sE_{1b}\rangle=\sA_{0}=\sA_{1}=\cdots=\sA_{9}\subset D^{b}(\widetilde{\hcoY}),\\
\langle\sF_{1b}^{*},\sF_{1a}^{*},\sF_{2}^{*},\sF_{3}^{*}\rangle=\sB_{0}=\sB_{1}=\cdots=\sB_{4}\subset D^{b}(\hchow).\end{aligned}
\end{matrix}\]

\begin{thm}[{\cite[Theorem 3.4.5, 8.1.1]{HoTa3}}]
\label{thm:Lefshetz-collections-Reye} Denote by $\sA_{i}(a)$,$\sB_{i}(b)$
the twists of $\sA_{i},\sB_{i}$ by $\sO_{\widetilde{\hcoY}}(aM_{\widetilde{\hcoY}})$
and $\sO_{\hchow}(bH_{\hchow})$, respectively. Then

{\rm (i)} $\langle\sA_{0},\sA_{1}(1),\cdots,\sA_{9}(9)\rangle$ is
a Lefschetz collection in $D^{b}(\widetilde{\hcoY})$, and 

{\rm (ii)} $\langle\sB_{4}(-4),\cdots,\sB_{1}(-1),\sB_{0}\rangle$
is a dual Lefschetz collection in $D^{b}(\hchow)$.

\noindent In particular the following vanishings hold:\[
\Hom_{D^{b}(\widetilde{\hcoY})}^{\bullet}(\sA_{i}(i),\sA_{j}(j))=0\;(i>j),\;\;\;\Hom_{D^{b}(\hchow)}^{\bullet}(\sB_{i}(-i),\sB_{j}(-j))=0\;(i<j).\]

\end{thm}
Although it is implicit in the above theorem, the (dual) Lefschetz
collections (i) and (ii) above indicate that there exist some non-commutative
resolutions of $\hcoY$ and $\chow$, respectively, and furthermore,
they are expected to be HPD with each other. This should be contrasted
to Theorem \ref{thm:HPD-Gr-Pf} where non-commutative resolution has
appeared only for the Pfaffian variety $\sY$. Of course, this difference
is due to the fact that both $\hcoY$ and $\chow$ are singular varieties
in our case. See \cite{Ku3} for a recent survey about known examples
of HPDs. 

As in Subsection \ref{sub:Derived-equivalence-Gr-Pf}, the derived
equivalence follows from the flatness of the ideal sheaf $I=\sI\vert_{Y\times X}$
over $X$ and the vanishing properties in Theorem \ref{thm:Ix-Iy-Thm}. 
\begin{thm}[{\cite[Theorem 8.0.3]{HoTa4}}]
 The restriction $I=\sI\vert_{Y\times X}$ defines a scheme $\Cs$
flat over $X$, and an equivalence $\Phi_{I}:D^{b}(Y)\to D^{b}(X)$
with $\Phi_{I}(-)=R\pi_{X*}(L\pi_{Y}^{*}(-)\otimes I)$. 
\end{thm}
The proof given in \cite[Sect.8]{HoTa4} proceeds in a similar way
to \cite{BC} and only uses the vanishing properties in Theorem \ref{thm:Lefshetz-collections-Reye}.

\vskip0.5cm

\subsection{BPS numbers }

The ideal sheaf $I$ describes a family of curves on $Y$ parametrized
by $x\in X$. In particular, in \cite{HoTa4}, an interesting relation
of them to some BPS number of $Y$ has been observed. Here we start
with the following proposition:
\begin{prop}[{\cite[Sect.3, Prop.7.2.2]{HoTa4}}]
 The ideal sheaf $I=\sI\mid_{Y\times X}$ defines a flat family $\left\{ C_{x}\right\} _{x\in X}$
whose general members are smooth curves of genus 3 and degree 5 in
$Y$. 
\end{prop}
The curve $C_{x}$ appears from the incidence relation $\Delta_{0}$
in $\rG(3,V)\times\rG(2,V)$. Recall $X=\hchow\cap\mP(L_{5}^{\perp})$
and the morphism $g:\hchow\to\rG(2,V)$. Then $g(x)\,(x\in X)$ determines
a line $l_{x}=\mP(V_{2,x})$. Then we have \[
\Delta_{0}\vert_{\rG(3,V)\times\left\{ g(x)\right\} }=\left\{ [\Pi]\in\rG(3,V)\mid l_{x}\subset\mP(\Pi)\right\} .\]
Now let us recall the definition of $\hcoY$ in (\ref{eq:def-hcoY-Stein-fact})
and $Y=\hcoY\cap\mP(L_{5})$. We define \[
\Zpq_{x}:=\left\{ ([Q],[\Pi])\mid l_{x}\subset\mP(\Pi)\subset Q\right\} \subset\Zpq\]
and \[
\gamma_{x}:=\Zpq_{x}\cap\pi_{\Zpq}^{-1}(Y)=\left\{ ([Q],[\Pi])\mid l_{x}\subset\mP(\Pi)\subset Q,[Q]\in\mP(L_{5})\right\} .\]
 When $Y$ is smooth, then $Y=\widetilde{\hcoY}\cap\mP(L_{5})=\hcoY\cap\mP(L_{5})$,
i.e., $\rho_{\hcoY}^{-1}(\mP(L_{5}))$ is away from the singular locus
$Sing(\hcoY)=\Hes_{2}$. On the other hand, over $\hcoY\setminus Sing(\hcoY)$
the Stein factorization $\Zpq\to\hcoY$ has the structure of a conic
bundle which is isomorphic to the generically conic bundle $\Zpq_{2}\to\hcoY_{2}$
over $\hcoY_{2}\setminus(\tilde{\rho}_{2}\circ\rho_{\widetilde{\hcoY}})^{-1}(Sing(\hcoY))$
(see {[}ibid,Sect.2.3{]} and also the next section). Therefore we
have $C_{x}=\pi_{\Zpq}(\gamma_{x})$ for the family of curves on $Y$.
We can further study the following properties:
\begin{prop}
\label{pro:shadow-curve}{\rm (1)} $\bar{\gamma}_{x}=\rho_{\hcoY}\circ\pi_{\Zpq}(\gamma_{x})=\rho_{\hcoY}(C_{x})$
is a plane quintic curve in $H=\Hes\cap\mP(L_{5})$ with 3 nodes and
arithmetic genus 6 for general $x\in X$. 

\begin{myitem}

\item[{\rm (2)}]  When $x\in X$ is general, $\bar{\gamma}_{x}$
is away from the branch locus $C_{H}\subset H$ and $C_{x}\to\bar{\gamma}_{x}$
is the normalization map. 

\item[{\rm (3)}] For general $x\in X$, there exists a 'shadow' curve
$C_{x}'$ of genus 3 and degree 5 with the properties $\rho_{\hcoY}^{-1}(\bar{\gamma}_{x})=C_{x}\cup C_{x}'$
and $C_{x}\cap C_{x}'=\rho_{\hcoY}^{-1}(\text{3 nodes of }\bar{\gamma}_{x})$. 

\end{myitem}
\end{prop}
We refer to {[}ibid Sect. 3, Fig.1{]} for details, but only remark
that the plane curve $\bar{\gamma}_{x}$ can be written explicitly
by $\bar{\gamma}_{x}=\left\{ [Q]\in H\mid l_{x}\subset Q\right\} $.
Considering the condition $l_{x}\subset Q$ under $x\in\hchow\cap\mP(L_{5}^{\perp})$,
we see easily that $\bar{\gamma}_{x}$ is a plane curve $H\cap P_{x}$
with \[
P_{x}=\left\{ [a_{ij}]\in\mP(L_{5})\mid\,^{t}zAz=\,^{t}wAw=0\,(\forall[z],[w]\in l_{x})\right\} \simeq\mP^{2},\]
where $A=(a_{ij})$ is the symmetric matrix corresponding to a point
$[a_{ij}].$ Note that $x\in\hchow\cap\mP(L_{5}^{\perp})$ implies
$\,^{t}zAw=0$, which is one of the three conditions for $l_{x}\subset Q$.
We depict the claims in Proposition \ref{pro:shadow-curve} in Fig.~4.1.

\begin{figure}
\begin{centering}
\includegraphics[scale=0.6]{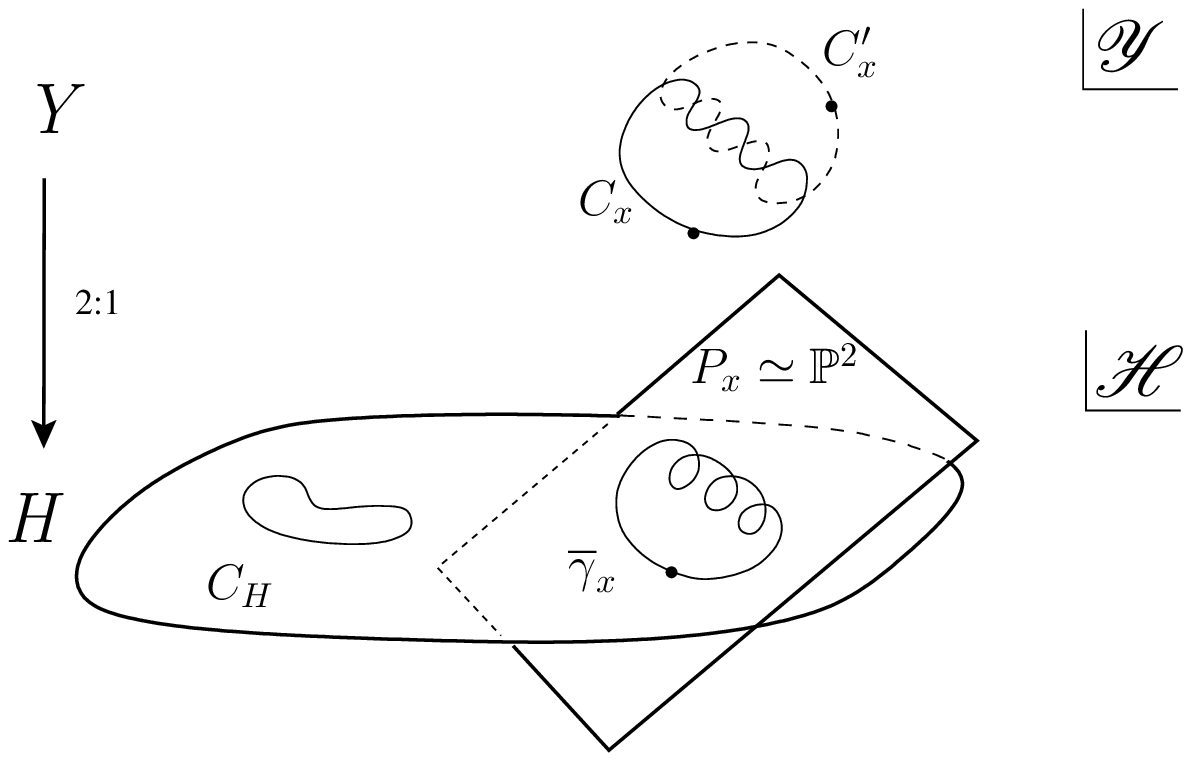}
\par\end{centering}

\begin{fcaption}\item{} \textbf{Fig.4.1. Shadow curve $C_{x}'$.}
Two intersecting curves $C_{x}$ and $C_{x}'$ in $Y$ covers the
plane quintic curve $\overline{\gamma}_{x}$ in $H$. $C_{H}$ is
the curve of the branch locus.

\end{fcaption}
\end{figure}

As claimed in Proposition \ref{pro:shadow-curve}, there are two (distinct)
families of curves $\left\{ C_{x}\right\} _{x\in X}$ and $\left\{ C_{x}'\right\} _{x\in X}$
in $Y$ parametrized $X$. These two are smooth curves of genus 3
and degree 5 for general $x\in X$, and interestingly, can be identified
in the BPS numbers calculated in {[}HoTa,1{]}. The relevant part of
the table of BPS numbers reads as follows: \begin{equation}
\begin{tabular}{c|cccccc}
 \mbox{\ensuremath{g}}  &  0  &  1  &  2  &  3  &  4  &  5 \\
\hline \mbox{\ensuremath{n_{g}^{Y}(d)}}  &  12279982850  &  571891188  &  3421300  &  100  &  0  &  0 \end{tabular}\label{eq:Reye-Y-BPS}\end{equation}
with $d=5$. As discovered in {[}ibid{]}, we can exactly identify
the two families in the BPS number $n_{3}^{Y}(5)=100$ as \[
n_{3}^{Y}(5)=(-1)^{\dim X}e(X)\times2=-(-50)\times2\]
following the counting {}``rule'' described in Subsection \ref{sub:BPS-numbers-Gr-Pf}.
This indicates that the BPS numbers, which are preferred in physics
interpretations \cite{GV} to other mathematical invariants such as
Donaldson-Thomas invariants, has a nice moduli interpretation in some
cases although their mathematical definition (as invariants of manifolds)
is difficult in general \cite{HST}. 

\vskip0.5cm

\subsection{\label{sub:Mirror-symmetry-Reye}Mirror symmetry }

In Subsection \ref{sub:Mirror-symmetry-Gr-Pf}, we have only described
the monodromy properties of Picard-Fuchs differential equation for
the mirror family of R{\o}dland's Pfaffian Calabi-Yau 3-fold. This
is partially because the geometry of the mirror family is rather involved.
Our second example of FM partners $\left\{ X,Y\right\} $ of $\rho=1$
has a nice feature from this perspective. We have a rather simple
description for the mirror family of Reye congruence Calabi-Yau 3-folds
$X$ in terms of special form of determinantal quintic hypersurfaces
in $\mP^{4}$. 

Recall the definition $X=\tS^{2}\mP(V)\cap\mP(L_{5}^{\perp})\subset\mP(\tS^{2}V)$.
Using the fact $\tS^{2}\mP(V)=\mP(V)\times\mP(V)/\mZ_{2}$, it is
easy to see the isomorphism $X\simeq\tilde{X}/\mZ_{2}$ with \begin{equation}
\tilde{X}=\left(\begin{matrix}\mP^{4}|\,1\,1\,1\,1\,1\\
\mP^{4}|\,1\,1\,1\,1\,1\end{matrix}\right)^{2,52},\label{eq:tilde-X-def}\end{equation}
where the superscripts $2,52$ represent the Hodge numbers $h^{1,1}$and
$h^{2,1}$, respectively. The r.h.s of (\ref{eq:tilde-X-def}) is
a common notation in physics literatures to represent complete intersections
of five (generic) $(1,1)$-divisors in $\mP^{4}\times\mP^{4}$. In
our case, we should read this as the complete intersection of five
generic and symmetric $(1,1)$-divisors which correspond to five linear
forms in $\mP(\tS^{2}V)$ determined by $L_{5}\subset\tS^{2}V^{*}$.
Note that when $L_{5}$ is taken in general position, $X$ is smooth
which means that the $\mZ_{2}$ action on $\tilde{X}$ is free. 

For concreteness, let us take a basis of $L_{5}$ by $A_{k}=(a_{ij}^{(k)})$
$(k=1,..,5).$ Then the defining equations of $\tilde{X}$ are given
by $f_{1}=f_{2}=...=f_{5}=0$ with $f_{k}=\sum_{i,j}z_{i}a_{ij}^{(k)}w_{j}$
and $([z],[w])\in\mP^{4}\times\mP^{4}$. If we introduce a notation
$A(z)=\left(\sum_{i}z_{i}a_{ij}^{(k)}\right)_{1\leq k,j\leq5}$ for
the $5\times5$ matrix defined by $A_{k}$, then we have \[
\tilde{X}=\left\{ ([z],[w])\in\mP^{4}\times\mP^{4}\mid A(z)w=0\right\} .\]
It is easy to deduce that the projection of $\tilde{X}$ to the first
factor of $\mP^{4}\times\mP^{4}$ is a determinantal quintic hypersurface,
\[
Z=\left\{ [z]\in\mP^{4}\mid\det\, A(z)=0\right\} .\]

\begin{prop}[\cite{HoTa1}]
{\rm (1)} When the linear subspace $L_{5}\subset\tS^{2}V^{*}$ is
general, the quintic hypersurface $Z$ is singular at $50$ ordinary
double points(ODPs) where $\rk A(z)=3$. {\rm (2)} The morphism $\pi_{1}:\tilde{X}\to Z$
is a small resolution of the 50 ODPs. 
\end{prop}
Details can be found in {[}ibid, Prop.3.3{]}. Here we summarize properties
of $X,\tilde{X}$ and $Z$ in the left of the following diagrams:
\def\tildeXZ{
\begin{xy}
(0,0)*+{\tilde X}="tX",
(0,-13)*+{X}="X",
(17,-13)*+{Z}="Z",
   (40,0)*+{\tilde{X}^*}="tXs",
   (40,-13)*+{\;X^*}="Xmirror",
  (52,0)*+{\tilde{X}_{sp}}="tXsp",
   (65,-13)*{Z_{sp}}="Zsp",
\ar_{/_{{\mathbb Z}_2}} "tXs";"Xmirror",
\ar_{/_{{\mathbb Z}_2}} "tX";"X",
\ar^{\,_{50 \text{OPD's}}} "tX";"Z",
\ar "tXs";"tXsp",
\ar "tXsp";"Zsp",
\end{xy} }\begin{equation}
\begin{matrix}\tildeXZ\end{matrix}\label{eq:diagram-X-mirror-Z}\end{equation}
For the construction of mirror family of $X$, we invoke the orbifold
mirror construction, which schematically described in the right diagram
of (\ref{eq:diagram-X-mirror-Z}). Namely, we start with a certain
special form $A_{sp}(z)$ of $A(z)$ (or the linear subspace $L_{5})$
to define $Z_{sp}=\left\{ \det\, A_{sp}(z)=0\right\} .$ $Z_{sp}$
is singular in general, and so is $\tilde{X}_{sp}:=\left\{ A_{sp}(z)w=0\right\} \subset\mP^{4}\times\mP^{4}.$
Finding a suitable crepant resolution $\tilde{X}^{*}\to\tilde{X}_{sp}$,
which is compatible with the $\mZ_{2}$ action of exchanging the two
factors of $\mP^{4}\times\mP^{4}$, we obtain a mirror family of $X$
by the quotient $X^{*}=\tilde{X}^{*}/\mZ_{2}$. In the final process,
we usually need to find a suitable finite group $G_{orb}$ (called
orbifold group) to arrive at the desired properties $h^{1,1}(X)=h^{2,1}(X^{*}$)
and $h^{2,1}(X)=h^{1,1}(X^{*})$, however interestingly it turns out
that $G_{orb}=\left\{ \id\right\} $ in our case. 

The special form $A_{sp}(z)$ found in \cite{HoTa2} corresponds to
a linear subspace $L_{5}=\langle A_{1},A_{2},\cdots,A_{5}\rangle$
with $A_{1},A_{2},...,A_{5}$ in order given by \[
\begin{aligned}\left(\begin{smallmatrix}1 & a & 0 & 0 & 0\\
a & 0 & 0 & 0 & 0\\
0 & 0 & 0 & 0 & 0\\
0 & 0 & 0 & 0 & 0\\
0 & 0 & 0 & 0 & 0\end{smallmatrix}\right),\;\left(\begin{smallmatrix}0 & 0 & 0 & 0 & 0\\
0 & 1 & a & 0 & 0\\
0 & a & 0 & 0 & 0\\
0 & 0 & 0 & 0 & 0\\
0 & 0 & 0 & 0 & 0\end{smallmatrix}\right),\;\left(\begin{smallmatrix}0 & 0 & 0 & 0 & 0\\
0 & 0 & 0 & 0 & 0\\
0 & 0 & 1 & a & 0\\
0 & 0 & a & 0 & 0\\
0 & 0 & 0 & 0 & 0\end{smallmatrix}\right),\;\left(\begin{smallmatrix}0 & 0 & 0 & 0 & 0\\
0 & 0 & 0 & 0 & 0\\
0 & 0 & 0 & 0 & 0\\
0 & 0 & 0 & 1 & a\\
0 & 0 & 0 & a & 0\end{smallmatrix}\right),\;\left(\begin{smallmatrix}0 & 0 & 0 & 0 & a\\
0 & 0 & 0 & 0 & 0\\
0 & 0 & 0 & 0 & 0\\
0 & 0 & 0 & 0 & 0\\
a & 0 & 0 & 0 & 1\end{smallmatrix}\right).\end{aligned}
\]
Using these special form of $A_{k}$, we have $Z_{sp}(a):=\left\{ \det A_{sp}(z)=0\right\} \subset\mP^{4}$
where\begin{equation}
\begin{matrix}\begin{aligned}\det A_{sp}(z)=\left|\begin{smallmatrix}z_{1}+az_{2} & az_{1} & 0 & 0 & 0\\
0 & z_{2}+az_{3} & az_{2} & 0 & 0\\
0 & 0 & z_{3}+az_{4} & az_{3} & 0\\
0 & 0 & 0 & z_{4}+az_{5} & az_{4}\\
az_{5} & 0 & 0 & 0 & z_{5}+az_{1}\end{smallmatrix}\right|\hskip3.2cm\\
=a^{5}z_{1}z_{2}z_{3}z_{4}z_{5}+(z_{1}+az_{2})(z_{2}+az_{3})(z_{3}+az_{4})(z_{4}+az_{5})(z_{5}+az_{1}).\end{aligned}
\end{matrix}\label{eq:Zsp-Det}\end{equation}
 By coordinate change, it is easy to see that $\left\{ Z_{sp}(a)\right\} _{a}$
defines a family of Calabi-Yau threefolds over $\mP^{1}$ by $[-a^{5},1]\in\mP^{1}$.
\begin{prop}
\begin{myitem} 

\item[{\rm (1)}]  When $a^{5}$ is general $(a^{5}\not=-\frac{1}{32},1-11a^{5}+a^{10}\not=0)$,
$Z_{sp}(a)$ is singular along 5 lines of $A_{2}$ singularities and
10 lines of $A_{1}$ singularities. 

\item[{\rm (2)}] $\tilde{X}_{sp}(a):=\left\{ ([z],[w])\mid A_{sp}(z)z=0\right\} $
partially resolves the singularities in {\rm (1)} to 20 lines of
$A_{1}$ singularities. 

\item[{\rm (3)}] There exists a crepant resolution $\tilde{X}^{*}(a)\to\tilde{X}_{sp}(a)$.
And $\tilde{X}^{*}(a)$ for general $a^{5}$ is a smooth Calabi-Yau
3-fold with Hodge numbers $h^{1,1}=52$, $h^{2,1}=2$. 

\end{myitem}
\end{prop}
More details of the singularities and their resolutions can be found
in {[}ibid{]}. For general $a^{5}$, we can see that $\tilde{X}^{*}(a)$
admits a free $\mZ_{2}$ action, and hence $X^{*}(a)=\tilde{X}^{*}(a)/\mZ_{2}$
is a Calabi-Yau 3-fold with Hodge numbers $h^{1,1}=26$, $h^{2,1}=1$.
We have then a family $\frak{X}^{*}:=\left\{ X^{*}(a)\right\} _{[-a^{5},1]\in\mP^{1}}$
of Calabi-Yau 3-folds over $\mP^{1}$.
\begin{prop}[{\cite[Prop.6.9]{HoTa2}}]
 $\frak{X}^{*}$ is a mirror family of Reye congruence Calabi-Yau
3-fold $X$.
\end{prop}
We omit the monodromy calculations which correspond to those in Subsection
\ref{sub:Mirror-symmetry-Gr-Pf}, since they are reported in {[}ibid,
Prop.2.10{]}. 

$\;$

\noindent \textbf{Remark. }(1) Set $x=-a^{5}$, then from the defining
equation (\ref{eq:Zsp-Det}) we observe that both $x=0$ and $x=\infty$
are MUM points. In \cite{HoTa1}, Gromov-Witten invariants $(g\leq14$)
of Reye congruence $X$ have been calculated from the MUM degeneration
at $x=0$ and the invariants of Fourier-Mukai partner $Y$ from $x=\infty$.
We believe that our mirror family $\frak{X}^{*}$ provides us a nice
example to study the geometry of mirror symmetry \cite{SYZ,GrS1,GrS2,RuS}
when non-trivial Fourier-Mukai partners exist. It is interesting,
although accidental, that in (\ref{eq:Zsp-Det}) we come across to
the geometry of quintic from which the study of mirror symmetry started
\cite{Gepner,GP,CdOGP}. 

(2) If we focus on the form of Picard-Fuchs differential operators
in \cite{AESZ,ES},\cite{DM}, there are many other examples which
exhibit two MUM points. Among them, a nice example has been identified
in \cite{Mi} with the mirror family of the Calabi-Yau 3-fold given
by general linear sections of a Schubert cycle in the Cayley plane
$E_{6}/P_{1}$. It is expected that this Calabi-Yau 3-fold has a non-trivial
Fourier-Mukai partner {[}ibid{]}\cite{Ga}. Also the mirror family
of the Calabi-Yau 3-folds given by the intersection of two copies
of Grassmannians $X=\rG(2,5)\cap\rG(2,5)\subset\mP^{9}$ \cite{Kan,Kap}
shows two MUM points whose interpretation seems slightly different
from those we have seen in this article. The two MUM points seems
to correspond Fourier-Mukai partners which are diffeomorphic but not
bi-holomorphic. It would be interesting to investigate these new examples
in more detail.

(3) In \cite{Hori}, the pair of Reye congruence Calabi-Yau 3-fold
$X$ and its Fourier-Mukai partner $Y$ have been understood in the
language of Gauged linear sigma modes along the arguments used for
the Grassmannian-Pfaffian example. Extending these arguments, many
other examples have been worked out in \cite{HoriKnapp} by calculating
the so-called {}``two sphere partition'' in physics \cite{JKLMR}. 

\newpage

\section{\label{sec:Birational-Geometry-of-hcoY}\textbf{\textup{Birational
Geometry of the Double Symmetroid $\hcoY$}}}

We describe the birational geometry of the double (quintic) symmetroid
$\hcoY$ and its resolution $\widetilde{\hcoY}$. We will see intensive
interplay of the projective geometry of quadrics and that of relevant
Grassmannians. In this section, we fix $V=\mC^{5}$ and retain all
the notations introduced in the last section. This section is an exposition
of the results whose details are contained in \cite{HoTa3,HoTa4}. 

\vskip0.5cm

\subsection{\label{sub:ZY-conic-1}Generically conic bundle $\Zpq\to\hcoY$}

We describe the (connected) fibers of $\Zpq\to\hcoY$ of the Stein
factorization $\Zpq\to\hcoY\to\Hes$ in (\ref{eq:def-hcoY-Stein-fact}).
Recall the definition $\Hes=\left\{ [a_{ij}]\in\mP(\tS^{2}V^{*})\mid\det a=0\right\} $
and \[
\Zpq:=\left\{ ([Q],[\Pi])\mid\mP(\Pi)\subset Q\right\} \subset\Hes\times\rG(3,V),\]
i.e., $\Zpq$ consists of pairs of singular quadric and (projective)
plane therein. The notation $[Q]\in\Hes$ above indicates that we
identify points $[a_{ij}]\in\Hes$ with the corresponding quadrics
$Q$ in $\mP(V)$. Since $\dim\Zpq-\dim\Hes=1$, we have generically
one dimensional fibers for $\pi_{\Zpq}:\Zpq\to\hcoY$. It is easy
to deduce the fibers of $\pi_{\Zpq}:\Zpq\to\hcoY$ from those of $\Zpq\to\Hes$:

The fibers of $\Zpq\to\Hes$ over a point $[Q]$ consists of planes
contained in the quadric $Q$. In Fig.~5.1 , depending on the rank
of $[Q]=[a_{ij}]$, the corresponding quadric $Q$ is depicted schematically.
Let us define reduced quadric $\bar{Q}$ to be the smooth quadric
naturally defined in $\mP(V/\Ker(a_{ij})$). Then, as is clear in
Fig.~5.1, $\bar{Q}\simeq\mP^{1}\times\mP^{1}$, a smooth conic, two
points and one point depending on $\rk Q$=4, 3, 2 and 1, respectively.
Singular quadrics $Q$ are then described by the cones over the reduced
quadric $\bar{Q}$ with the vertex $\Ker Q:=\mP(\Ker(a_{ij}))$. The
fibers of $\pi_{\Zpq}:\Zpq\to\hcoY$ over $y\in\hcoY$ are given by
connected families of planes contained in the quadric $Q_{y}=\rho_{\hcoY}(y)$.
We summarize the connected fibers:

$\;$

\begin{myitem3} 

\item{(a)} When $\rk Q_{y}=4$, the fiber is the $\mP^{1}$-families
of planes which corresponds to one of the two possible rulings of
$\bar{Q}_{y}\simeq\mP^{1}\times\mP^{1}$. 

\item{(b)} When $\rk Q_{y}=3$, the fiber is the $\mP^{1}$-family
of planes parametrized by the conic $\bar{Q}_{y}.$

\item{(c)} When $\rk Q_{y}=2$, the fiber is the planes parametrized
by $(\mP^{3})^{*}\sqcup_{1pt}(\mP^{3})^{*}$ where $(\mP^{3})^{*}$
parametrizes planes in $\mP^{3}$ and $A\sqcup_{1pt}B$ represents
the union with $a\in A$ and $b\in B$ (one point from each) are identified.

\item{(d)} When $\rk Q_{y}=1$, the fiber is the planes parametrized
by $(\mP^{3})^{*}$.

\end{myitem3} 

$\;$

We remark that, in the case of (a), one of the two possible $\mP^{1}$-families
of planes is specified (by the definition of Stein factorization)
when we take $y\in\hcoY$. This and the other cases explain the finite
morphism $\rho_{\hcoY}:\hcoY\to\Hes$ which is $2:1$ over $\Hes_{4}\setminus\Hes_{3}$
and branched over $\Hes_{3}$. We say that a point $y\in\hcoY$ has
rank $i$ if $\rank a_{y}=i$ for $\rho_{\hcoY}(y)=[a_{y}]$, and
define $G_{\hcoY}:=\left\{ y\in\hcoY\mid\rk y\leq2\right\} $. Note
that $\dim G_{\hcoY}=\dim\Hes_{2}=8$.
\begin{prop}
\label{pro:conic-bundle-Z-Y}\begin{myitem} \item{\rm (1)} $\mathrm{Sing\,}\Hes=\Hes_{3}$
and $\mathrm{Sing}\,\hcoY=G_{\hcoY}(=\Hes_{2}).$

\item{\rm (2)} $\pi_{\Zpq}:\Zpq\to\hcoY$ is a generically conic
bundle with the conics in $\rG(3,V)$.

\end{myitem}\end{prop}
\begin{proof}
(1) $\mathrm{Sing\,}\Hes=\Hes_{3}$ follows from the basic properties
of secant varieties. For the latter claim $\mathrm{Sing}\hcoY=G_{\hcoY}$,
we refer to \cite[Prop.5.7.2]{HoTa3}. 

\noindent (2) Over $\hcoY\,\setminus G_{\hcoY}$, the fibers of $\pi_{\Zpq}:\Zpq\to\hcoY$
consists of smooth $\mP^{1}$-families of planes in $\rG(3,V)$. As
we see in the next subsection, it is easy to see that these are smooth
conics on $\rG(3,V)$.
\end{proof}
\vskip0.5cm%
\begin{figure}
\begin{centering}
\includegraphics[scale=0.7]{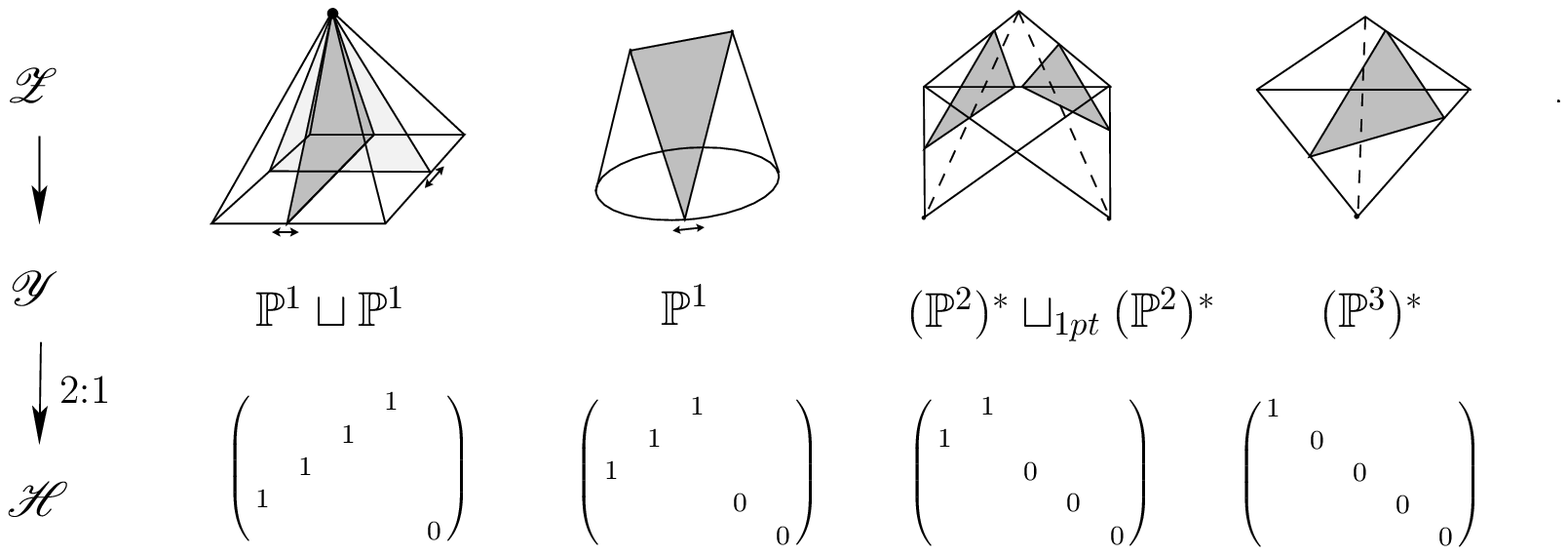}
\par\end{centering}

\begin{fcaption}\item{} \textbf{Fig.5.1. Quadrics and planes therein.}
Quadrics $Q$ are depicted for each rank, $\rk Q=4,3,2,1$. When $\rk Q=4$,
there are two connected fibers of $\Zpq\to\Hes$. 

\end{fcaption}
\end{figure}

\subsection{\label{sub:Birational-model-barY}Birational model $\overline{\hcoY}$
of $\hcoY$ }

Let us consider a quadric $Q$ of rank 4 and 3, in order, and a $\mP^{1}$-family
of planes in $Q$. 

First, for a quadric $Q$ of rank 4, let us denote the vertex of $Q$
(the kernel of $(a_{ij}))$ by $\langle v\rangle$. Then, one of the
$\mP^{1}$-family of plane described in (a) in Subsection \ref{sub:ZY-conic-1}
takes the following form:\[
\left\{ [\Pi_{s,t}]\right\} :=\left\{ \langle c(s,t),d(s,t),v\rangle\mid[s,t]\in\mP^{1}\right\} ,\]
where $c(s,t),d(s,t)\in V$ are linear in $s,t$ and span the $\langle c(s,t),d(s,t)\rangle\simeq\mP^{1}$
which gives the ruling $\bar{Q}\simeq\mP^{1}\times\mP^{1}$. One of
the key observations is that for such a $\mP^{1}$-family of plane
we have a conic $q$ in $\mP(\wedge^{3}V)$ by \[
q:=\left\{ [c\wedge d\wedge v]=[\Lambda_{0}s^{2}+\Lambda_{1}st+\Lambda_{2}t^{2}]\mid[s,t]\in\mP^{1}\right\} ,\]
which actually defines a conic in $\rG(3,V)$ by the Pl\"ucker embedding
$\rG(3,V)\subset\mP(\wedge^{3}V)$. We note that conic $q$ resides
in the plane $\mP_{q}$ which is uniquely determined by the $\mP^{1}$-family,
\[
\mP_{q}:=\langle\Lambda_{0},\Lambda_{1},\Lambda_{2}\rangle\subset\mP(\wedge^{3}V).\]

When $\rk Q=3$, we start with $\left\{ [\Pi_{s,t}]\right\} =\left\{ \langle d(s,t),v_{1},v_{2}\rangle\mid[s,t]\in\mP^{1}\right\} $
with $v_{1},v_{2}$ being bases of $\Ker(a_{ij})$ and $d(s,t)=s^{2}v_{3}+stv_{4}+t^{2}v_{5}$
parametrizing the conic $\bar{Q}$ in $\mP(V/\Ker(a_{ij}))$. Again,
we have the corresponding conic $q$ in $\rG(3,V)$ and also the plane
$\mP_{q}\subset\mP(\wedge^{3}V)$ which contains the conic $q$. 

The conics $q$ above explain the generically conic bundle $\Zpq\to\hcoY$
claimed in Proposition \ref{pro:conic-bundle-Z-Y}. The planes $\mP_{q}\subset\mP(\wedge^{3}V)$
and conics $q$ will play central roles in the description of the
resolution $\widetilde{\hcoY}\to\hcoY$. Here noting that the planes
$\mP_{q}$ above have a specific forms, we define the following subset
of planes in $\mP(\wedge^{3}V)$: \[
\overline{\hcoY}=\left\{ [U]\in\rG(3,\wedge^{3}V)\mid U=\bar{U}\wedge v\,\text{\text{ for some }}v\in\mP(V)\right\} ,\]
where we regard $\bar{U}$ as an element in $\mP(\wedge^{2}(V/V_{1}))$
with $V_{1}=\mC v$. To introduce a (reduced) scheme structure on
the subset $\overline{\hcoY}$, we consider a linear morphism $\varphi:\tS^{2}(\wedge^{3}V)\to V$
by the composition of the following natural linear morphisms: \begin{equation}
\varphi:\tS^{2}(\wedge^{3}V)\to\tS^{2}(\wedge^{2}V^{*})\to\wedge^{4}V^{*}\simeq V.\label{eq:map-varphi-U}\end{equation}
We define $\varphi_{U}:=\varphi\vert_{\tS^{2}U}$ to be the natural
restriction of $\varphi$ for a fixed subspace $[U]\in\rG(3,\wedge^{3}V)$.
Then, we have the following proposition: 
\begin{prop}
\begin{myitem} \item{\rm (1)} $U\subset\wedge^{3}V$ decomposes as
$U=\bar{U}\wedge v$ if and only if $\rk\varphi_{U}\leq1$. 

\noindent \item{\rm (2)} The scheme $\left\{ [U]\in\rG(3,\wedge^{3}V)\mid\rk\varphi_{U}\leq1\right\} $
is nonreduced along the singular locus of its reduced structure. 

\end{myitem}
\end{prop}
The proof of the above proposition follows by writing the rank condition
explicitly for the matrix representing $\varphi_{U}$ under suitable
bases (see \cite[Subsect.5.3, 5.4]{HoTa3}). Hereafter, we consider
$\overline{\hcoY}$ as the scheme with the reduced structure on $\left\{ [U]\in\rG(3,\wedge^{3}V)\mid\rk\varphi_{U}\leq1\right\} $. 
\begin{prop}
\label{pro:birational-bY-Y}$\hcoY$ and $\overline{\hcoY}$ are birational. \end{prop}
\begin{proof}
By definition of the Stein factorization, points $y\in\hcoY$ are
specified by the connected fibers of $\Zpq\to\hcoY$, which are generically
given by conics $q$ in $\rG(3,V)$. Hence we can write general points
$y\in\hcoY$ by $y=([Q_{y}],q_{y}$) where $[Q_{y}]=\rho_{\hcoY}(y)$
and the corresponding conic $q_{y}$ which is a $\mP^{1}$-family
of planes contained in $Q_{y}$. Rational map $\hcoY\dashrightarrow\overline{\hcoY}$
has been described already above by $y=([Q_{y}],q_{y})\to\mP_{q_{y}}$
for $y\in\hcoY\setminus G_{\hcoY}$. To describe the inverse rational
map $\overline{\hcoY}\dashrightarrow\hcoY$, we note that the following
isomorphism for $U=\bar{U}\wedge v\in\wedge^{3}V$:\begin{equation}
\mP(U)\cap\rG(3,V)\text{ in }\mP(\wedge^{3}V)\simeq\mP(\bar{U})\cap\rG(2,V/V_{1})\text{ in }\mP(\wedge^{2}(V/V_{1})),\label{eq:PintG(3)-and-PintG(2)}\end{equation}
where $V_{1}=\mC v$. Since $\rG(2,V/V_{1})\simeq\rG(2,4)$ is the
Pl\"ucker quadric, when $U$ is general, the r.h.s. determines a
smooth conic on $\rG(2,V/V_{1})$ and in turn a smooth conic on $\rG(3,V)$.
We can see that this is the inverse rational map. 
\end{proof}
\textcolor{red}{$\;$}

Obviously, the inverse rational map $\overline{\hcoY}\dashrightarrow\hcoY$
is not defined when $\mP(U)\cap\rG(3,V)=\mP(U)$, i.e. $\mP(U)\subset\rG(3,V)$.\textcolor{red}{{}
}There are two cases where $\mP(U)\subset\rG(3,V)$ occurs for $[U]\in\overline{\hcoY}$:
The first one is when $\mP(U)$ is given by the Pl\"ucker image of
the plane \[
{\tt P}_{V_{2}}:=\left\{ [\Pi]\mid V_{2}\subset\Pi\subset V\right\} \simeq\mP^{2}\]
in $\rG(3,V)$ for some $V_{2}$. The second one is given by the Pl\"ucker
image of the plane\[
{\tt P}_{V_{1}V_{4}}:=\left\{ [\Pi]\mid V_{1}\subset\Pi\subset V_{4}\right\} \simeq\mP^{2}\]
in $\rG(3,V)$ for some $V_{1}$ and $V_{4}$. The plans of the form
${\tt P}_{V_{2}}$ and ${\tt P}_{V_{1}V_{4}}$, respectively, are
called \textit{$\rho$-planes} and \textit{$\sigma$-planes. }These
planes determine the following loci in $\overline{\hcoY}$: \begin{equation}
\begin{aligned}\overline{\sP}_{\rho} & :=\left\{ [U]\mid V_{2}\subset V,\text{ }U=V/V_{2}\wedge(\wedge^{2}V_{2})\right\} ,\\
\overline{\sP}_{\sigma} & :=\left\{ [U]\mid V_{1}\subset V_{4}\subset V,\; U=\wedge^{2}(V_{4}/V_{1})\wedge V_{1}\right\} .\end{aligned}
\label{eq:Prho-Psig}\end{equation}
Note that $\overline{\sP}_{\rho}\simeq\rG(2,V)$ and $\overline{\sP}_{\sigma}\simeq F(1,4,V)$.

\vskip0.5cm

\subsection{Sing~$\overline{\hcoY}$ and resolutions of $\overline{\hcoY}$ }

We consider the reduced structure on $\overline{\hcoY}$ as described
in the preceding subsection. Then writing the condition $\rk\varphi_{U}\leq1$,
we can study the singularities of $\overline{\hcoY}$ explicitly. 
\begin{prop}
\label{pro:Resolutions-Y}\begin{myitem} \item{\rm (1)} $\overline{\hcoY}$
is singular along $\overline{\sP}_{\rho}\simeq\rG(2,V)$. 

\item{\rm (2)} Define $\hcoY_{3}:=\left\{ ([U],[V_{1}])\mid U=\bar{U}\wedge V_{1}\right\} \subset\overline{\hcoY}\times\mP(V)$,
then the natural projection $\hcoY_{3}\to\overline{\hcoY}$ is a resolution
of the singularity. 

\item{\rm (3)} $\hcoY_{3}$ is isomorphic to the Grassmannian bundle
$\rG(3,\wedge^{2}T_{\mP(V)}(-1))$ over $\mP(V)$. 

\item{\rm (4)} The singularities of $\overline{\hcoY}$ are the affine
cone over $\mP^{1}\times\mP^{5}$ along $\overline{\sP}_{\rho}$,
and there is a (anti-)flip to another resolution $\widetilde{\hcoY}\to\overline{\hcoY}$
which fits into the following diagram: \def\FigFlop{
\begin{xy} 
(-10,01)*+{ },
(43,-10)*+{\hcoY.}="Y", 
(14,2)*+{\hcoY_2}="Yii", 
(1,-10)*+{\hcoY_3}="Yiii", 
(5,-10)*+{}="yiii", 
(22,-10)*+{}="ty", 
(27,-10)*+{\widetilde{\hcoY}}="tY", 
(14,-22)*+{\overline{\hcoY}}="bY", 
(19,-20)*+{}="by", 
(35,-20)*+{}="h", 
(35,-8)*+{\,_{\rho_{\widetilde{\hcoY}}}}, 
\ar "Yii";"tY" 
\ar "Yii";"Yiii" 
\ar "Yiii";"bY" 
\ar  "tY";"Y" 
\ar "tY";"bY" 
\ar @{-->}^{\,_\text{(anti-)flip}} "yiii";"ty" 
\end{xy} }\begin{equation}
\begin{matrix}\FigFlop\end{matrix}\label{eq:Fig-anti-Flop}\end{equation}

\end{myitem}\end{prop}
\begin{proof}
(1) and (4) follow directly by writing the condition $\rk\varphi_{U}\leq1$,
see \cite[Prop.5.4.2, 5.4.3]{HoTa3}. Global descriptions of the blow-up
$\hcoY_{2}\to\hcoY_{3}$ will be given in Proposition \ref{pro:Y3-Y2-Frho}.
(3) We consider $\bar{U}\simeq\mC^{3}$ as a subspace in $\wedge^{2}(V/V_{1})$.
Then claim is clear since $T_{\mP(V)}(-1)\vert_{[V_{1}]}\simeq V/V_{1}$.
(2) follows from (3). 
\end{proof}
We denote by $\sP_{\rho}$ the exceptional set (which is contracted
to $\overline{\sP}_{\rho}$) of the resolution $\hcoY_{3}\to\overline{\hcoY}$
and by $\sP_{\sigma}\simeq\overline{\sP}_{\sigma}$ the proper transform
of $\overline{\sP}_{\sigma}$. It is easy to observe the following
isomorphisms:\begin{equation}
\sP_{\rho}\simeq F(1,2,V)\simeq\mathbb{P}(T_{\mP(V)}(-1)),\;\;\;\;\sP_{\sigma}\simeq F(1,4,V)\simeq\mathbb{P}(T_{\mP(V)}(-1)^{*}).\label{eq:Pr-Psigma}\end{equation}
These loci $\sP_{\rho}$ and $\sP_{\sigma}$ in $\rG(3,\wedge^{3}T_{\mP(V)}(-1))$
will be interpreted in the next section.

In the diagram (\ref{eq:Fig-anti-Flop}), we have included the content
of the following theorem:
\begin{thm}
\label{thm:Resolution-Y}There is a morphism $\rho_{\widetilde{\hcoY}}:\widetilde{\hcoY}\to\hcoY$
which contracts an exceptional divisor $F_{\widetilde{\hcoY}}$ to
the singular locus $G_{\hcoY}$ of $\hcoY$. 
\end{thm}
The above theorem is one of the main results of \cite{HoTa3}. We
refer to {[}ibid, Subsect.~5.7, and Fig.2{]} for details. Also, for
the proof of Theorem \ref{thm:Lefshetz-collections-Reye}, we used
a natural flattening of the fibers of $F_{\widetilde{\hcoY}}\to G_{\hcoY}$
constructed in {[}ibid,Section 7{]}. Below, we describe the construction
of the morphism $\rho_{\widetilde{\hcoY}}$ briefly. 

\vskip0.5cm

\subsection{The resolution $\rho_{\widetilde{\hcoY}}:\widetilde{\hcoY}\to\hcoY$ }

We formulate a rational map $\varphi_{DS}:\overline{\hcoY}\dashrightarrow\Hes$
which extends to $\tilde{\varphi}_{DS}:\widetilde{\hcoY}\to\Hes$.
Then the Stein factorization of $\tilde{\varphi}_{DS}$ gives the
claimed morphism $\rho_{\widetilde{\hcoY}}:\widetilde{\hcoY}\to\hcoY$
{[}ibid,Prop.5.6.1{]}. 

The key relation for the construction is the following decomposition:\begin{equation}
\wedge^{3}(\wedge^{2}(V/V_{1}))=\Sigma^{(3,1,1,1)}(V/V_{1})\oplus\Sigma^{(2,2,2,0)}(V/V_{1})\simeq\tS^{2}(V/V_{1})\oplus\tS^{2}(V/V_{1})^{*},\label{eq:spin-decomp}\end{equation}
as irreducible $so(\wedge^{2}V/V_{1})\simeq sl(V/V_{1})$-modules,
where $\Sigma^{\alpha}$ represents the Schur functor. We called this
\textit{double spin decomposition} since the r.h.s. is $V_{2\lambda_{s}}\oplus V_{2\lambda_{\bar{s}}}$
with the spinor and conjugate spinor weights $\lambda_{s}$ and $\lambda_{\bar{s}}$,
respectively. $\rG(3,\wedge^{2}(V/V_{1}))$ consists of 3-spaces in
$\wedge^{2}(V/V_{1})$. We have also $\OG(3,\wedge^{2}(V/V_{1}))$
which consists of isotropic 3-spaces with respect to the natural symmetric
form $\wedge^{2}(V/V_{1})\times\wedge^{2}(V/V_{1})\to\wedge^{4}(V/V_{1})\simeq\mC$.
We denote by $\OG^{\pm}(3,\wedge^{2}(V/V_{1}))$ the connected components
of $\OG(3,\wedge^{2}(V/V_{1}))$.

If we consider the above decomposition fiberwise for $\wedge^{2}T_{\mP(V)}(-1)$,
then we have the following embedding:\begin{equation}
\begin{matrix}\begin{aligned}i: & \hcoY_{3}=\rG(3,\wedge^{2}T_{\mP(V)}(-1))\\
 & \hookrightarrow\mP(\tS^{2}T_{\mP(V)}(-1)\otimes\wedge^{4}T_{\mathbb{P}(V)}(-1)\oplus\tS^{2}T_{\mP(V)}(-1)^{*}\otimes(\wedge^{4}T_{\mathbb{P}(V)}(-1))^{\otimes2}).\end{aligned}
\end{matrix}\label{eq:G36-spin-embed2}\end{equation}

\begin{prop}
\label{pro:spin-OG36}The following properties hold for the loci $\sP_{\rho}$
and $\sP_{\sigma}$ in $\hcoY_{3}$:

\begin{myitem} 

\item{\rm (1)} $i(\sP_{\rho})=v_{2}(\mathbb{P}(T_{\mP(V)}(-1)))$,
$i(\sP_{\sigma})=v_{2}(\mathbb{P}(T(-1)^{*}))$. 

\item{\rm (2)} $\sP_{\rho}=\OG^{+}(3,\wedge^{2}T_{\mP(V)}(-1))$,
$\sP_{\sigma}=\OG^{-}(3,\wedge^{2}T_{\mP(V)}(-1)^{*})$. 

\end{myitem}\end{prop}
\begin{proof}
(1) The claimed relations follow from the isomorphisms (\ref{eq:Pr-Psigma})
and the form of the embedding (\ref{eq:G36-spin-embed2}). We can
also verify the claim explicitly by writing the decomposition (\ref{eq:spin-decomp})
(see Appendix \ref{sec:AppendixB}). (2) The points $[V_{1},V_{2}]\in F(1,2,V)\simeq\sP_{\rho}$
determine the corresponding points $([\bar{U}],[V_{1}])\in\sP_{\rho}$
with $[\bar{U}]=[(V/V_{2})\wedge(V_{2}/V_{1})]\in\rG(3,\wedge^{2}(V/V_{1}))$.
Then we verify $\bar{U}\wedge\bar{U}=0$. Similarly, points $([\bar{U}],[V_{1}])\in\sP_{\sigma}$
have the forms $[\bar{U}]=[\wedge^{2}(V_{4}/V_{1})]$ for some $V_{4}$.
Again, we have $\bar{U}\wedge\bar{U}=0$. The claims follow since
all maximally isotropic subspaces in $\wedge^{2}(V/V_{1})$ take either
of these two forms. 
\end{proof}
Now we consider the following sequence of (rational) morphisms:\begin{equation}
\begin{matrix}\begin{aligned}\hcoY_{3}\overset{i}{\hookrightarrow}\mP(\tS^{2}T(-1)\otimes\mathcal{O}_{\mathbb{P}(V)}(1)\oplus\tS^{2}T(-1)^{*}\otimes\mathcal{O}_{\mathbb{P}(V)}(2))\qquad\qquad\\
\dashrightarrow\mP(\tS^{2}T(-1)^{*})\hookrightarrow\mP(\tS^{2}V^{*}\otimes\mathcal{O}_{\mathbb{P}(V)})\rightarrow\mathbb{P}(S^{2}V^{*}),\end{aligned}
\end{matrix}\label{eq:sequence-Y3-to-S2V}\end{equation}
where we use $\wedge^{4}T_{\mathbb{P}(V)}(-1)=\mathcal{O}_{\mathbb{P}(V)}(1)$,
and (here and hereafter) we write $T(-1)$ for $T_{\mP(V)}(-1)$ to
simplify formulas. In the middle, we consider the projection to the
second factor. The injection in the right is defined by considering
the dual of the surjection $V\otimes\sO_{\mP(V)}\to T(-1)\to0$, and
$\mP(\tS^{2}V^{*}\otimes\mathcal{O}_{\mathbb{P}(V)})\rightarrow\mathbb{P}(S^{2}V^{*})$
is the natural projection for $\mP(\tS^{2}V^{*}\otimes\mathcal{O}_{\mathbb{P}(V)})=\mP(\tS^{2}V^{*})\times\mathbb{P}(V)$.
Since the image of the composition is in $\Hes\subset\mP(\tS^{2}V^{*})$,
we have a rational map,\[
\phi_{DS}:\hcoY_{3}\dashrightarrow\Hes.\]

\begin{prop}
\begin{myitem} \item{\rm (1)} The rational map $\phi_{DS}$ defines
a morphism $\phi_{DS}:\hcoY_{3}\setminus\sP_{\rho}\simeq\overline{\hcoY}\setminus\overline{\sP}_{\rho}\to\Hes$.
In particular, it induces a rational map $\varphi_{DS}:\overline{\hcoY}\dashrightarrow\Hes$
whose indeterminacy locus is $\overline{\sP}_{\rho}$. 

\item{\rm (2)} $\phi_{DS}(\sP_{\sigma})=\varphi_{DS}(\overline{\sP}_{\sigma})=\Hes_{1}$.

\item{\rm (3)} The rational map $\varphi_{DS}:\overline{\hcoY}\dashrightarrow\Hes$
extends to a morphism $\tilde{\varphi}_{DS}:\widetilde{\hcoY}\to\Hes$. 

\end{myitem}\end{prop}
\begin{proof}
(1) and (2) follow from the claim (1) in Theorem \ref{pro:spin-OG36}
and the definition $\varphi_{DS}$ with the Pl\"ucker embedding (\ref{eq:G36-spin-embed2}).
We can verify (3) explicitly by writing the rational map $\varphi_{DS}$
and extending it to the blow-up $\widetilde{\hcoY}\to\overline{\hcoY}$
(see \cite[Prop.5.5.3]{HoTa3}).\end{proof}
\begin{thm}
$\tilde{\varphi}_{DS}:\widetilde{\hcoY}\to\Hes$ factors as $\widetilde{\hcoY}\to\hcoY\overset{\rho_{\hcoY}}{\to}\Hes$
with the morphism $\rho_{\hcoY}:\hcoY\to\Hes$ in (\ref{eq:def-hcoY-Stein-fact}).
This defines the resolution $\rho_{\widetilde{\hcoY}}:\widetilde{\hcoY}\to\hcoY$. \end{thm}
\begin{proof}
The claim basically follows from the Stein factorization. In {[}ibid,
Section 5.6, Fig.2{]}, the fibers of $\tilde{\varphi}_{DS}:\widetilde{\hcoY}\to\Hes$
have been described completely, and the claim is clear from the results
there.
\end{proof}
\noindent\textbf{Remark.} We describe the inverse image of the rational
map $\phi_{DS}$. Let us fix $[a]\in\Hes$. When we fix (a choice
of) $V_{1}\subset\Ker a$, we have a {}``reduced matrix'' $[a_{V_{1}}]\in\mP(\tS(V/V_{1})^{*})$
representing the quadric in $\mathbb{P}(V/V_{1})$. Consider the restriction
$\phi_{V_{1}}:=\phi_{DS}\vert_{\pi_{3}^{-1}([V_{1}])}$ of $\phi_{DS}$
to the fiber $\pi_{3}^{-1}([V_{1}])=\rG(3,\wedge^{2}(V/V_{1}))$ of
$\pi_{3}:\hcoY_{3}\to\mP(V)$, and also similar restriction $i_{V_{1}}:\rG(3,\wedge^{2}(V/V_{1}))\hookrightarrow\mP(\tS^{2}(V/V_{1})\oplus\tS^{2}(V/V_{1})^{*})$
of the Pl\"ucker embedding (\ref{eq:G36-spin-embed2}). Then, over
the fiber $\pi_{3}^{-1}([V_{1}])$,\textcolor{red}{{} }the rational
map $\phi_{DS}:\hcoY_{3}\dashrightarrow\Hes$ (\ref{eq:sequence-Y3-to-S2V})
is basically given by the projection $\mP(\tS^{2}(V/V_{1})\oplus\tS^{2}(V/V_{1})^{*})\dashrightarrow\mP(\tS^{2}(V/V_{1})^{*})$
sending $[v_{ij},w_{kl}]$ to $[w_{ij}]$. The ideal of the Pl\"ucker
embedding in terms coordinate $[v_{ij},w_{kl}]$ turns out to have
a rather  nice form as shown in Appendix \ref{sec:AppendixB}. Using
the results listed in Appendix \ref{sec:AppendixB}, we can prove
the following properties of the inverse image of $\phi_{DS}$: 

1) When $\rk a=4$, $V_{1}$ is unique and we have $i_{V_{1}}\circ\phi_{V_{1}}^{-1}(a)=[\pm\sqrt{\det a_{V_{1}}}\, a_{V_{1}}^{-1},a_{V_{1}}].$

2) When $\rk a=3$, for any $V_{1}\subset\Ker a$, we have $i_{V_{1}}\circ\phi_{V_{1}}^{-1}(a)=\emptyset$.

3) When $\rk a=2$, for each choice of $V_{1}\subset\Ker a$, we have
$i_{V_{1}}\circ\phi_{V_{1}}^{-1}(a)\simeq\mP^{1}\times\mP^{1}.$

4) When $\rk a=1$, \textcolor{black}{for} each choice of $V_{1}\subset\Ker a$,
we have $i_{V_{1}}\circ\phi_{V_{1}}^{-1}(a)\simeq\mP(1^{3},2).$

\noindent Let us denote by $G_{\rho}$ the exceptional set of the
resolution $\tilde{\varphi}_{DS}:\widetilde{\hcoY}\to\overline{\hcoY}$.
Then, since $\hcoY_{3}\setminus\sP_{\rho}\simeq\overline{\hcoY}\setminus\overline{\sP}_{\rho}\simeq\widetilde{\hcoY}\setminus G_{\rho}$,
we can identify $\phi_{DS}$, $\varphi_{DS}$ and $\tilde{\varphi}_{DS}$
with each other over these complement sets. Then the above results
indicate that $\tilde{\varphi}_{DS}^{-1}(a)\,(\rk a=3)$ is contained
in the exceptional set $G_{\rho}$ (and this is indeed the case {[}ibid,
Lemma 5.6.2{]}). Note also that from 3) and 4) and $\dim G_{\hcoY}=8\,(G_{\hcoY}\simeq\Hes_{2})$,
we see that $\tilde{\varphi}_{DS}^{-1}(G_{\hcoY})$ is a divisor in
$\widetilde{\hcoY}$, which is nothing but the divisor $F_{\widetilde{\hcoY}}$
that appeared in Theorem \ref{thm:Resolution-Y}. Full details of
1)--4) can be found in {[}ibid, Section 5.6{]} (see also {[}ibid,
Fig.2{]}). \hfill $\square$

\vskip0.5cm

\subsection{Generically conic bundles}

We describe the generically conic bundle $\pi_{2'}:\Zpq_{2}\to\hcoY_{2}$
which has appeared in (\ref{eq:diag-YZX}). The basic idea is the
same as that we used in the proof of Proposition \ref{pro:birational-bY-Y},
i.e., to consider the intersection $\mP(U)\cap\rG(3,V)\simeq\mP(\bar{U})\cap\rG(3,V/V_{1})$
for $U=\bar{U}\wedge V_{1}$.

\subsubsection{{\itshape\bfseries Generically conic bundle $\overline{\Zpq}\to\overline{\hcoY}$
}}

Let us fix the embedding $\rG(3,V)\subset\mP(\wedge^{3}V)$. We recall
the definition \[
\overline{\hcoY}=\left\{ [U]\in\rG(3,\wedge^{3}V)\mid U=\bar{U}\wedge V_{1}\text{ for some }V_{1}\subset V\right\} .\]
Then from the isomorphism (\ref{eq:PintG(3)-and-PintG(2)}), we have
generically conic bundle by \[
\overline{\Zpq}:=\left\{ ([c],[U])\mid[c]\in\mP(U)\cap\rG(3,V),[U]\in\overline{\hcoY}\right\} \subset\rG(3,V)\times\overline{\hcoY},\]
with the natural projection $\overline{\Zpq}\to\overline{\hcoY}$.
As explained in Subsection \ref{sub:Birational-model-barY}, the fibers
$\mP(U)\cap\rG(3,V)$ over point $[U]$ are conics for $[U]\in\overline{\hcoY}\setminus(\overline{\sP}_{\rho}\cup\overline{\sP}_{\sigma})$
while they are $\rho$-planes and $\sigma$-planes ($\simeq\mP(U)$)
for $[U]\in\overline{\sP}_{\rho}$ and $[U]\in\overline{\sP}_{\sigma}$,
respectively.

\subsubsection{{\itshape\bfseries Generically conic bundle $\Zpq_{3}\to\hcoY_{3}$
}}

The generically conic bundle $\overline{\Zpq}\to\overline{\hcoY}$
naturally extends to $\Zpq_{3}\to\hcoY_{3}$ by the isomorphism $\mP(U)\cap\rG(3,V)\simeq\mP(\bar{U})\cap\rG(2,V/V_{1})$
for $U=\bar{U}\wedge V_{1}$. To describe it, let us introduce the
universal bundles for the Grassmannian bundle $\pi_{3}:\hcoY_{3}=\rG(3,\wedge^{2}T(-1))\to\mP(V),$\[
0\to\sS\to\pi_{3}^{*}\wedge^{2}T(-1)\to\sQ\to0.\]
Denote by $\mP(S)$ the universal planes over $\hcoY_{3}$, whose
fiber over $([\bar{U}],[V_{1}])$ is $\mP(\bar{U})$. Now, consider
Grassmannian bundle $\pi_{G}:\rG(2,T(-1))\to\mP(V)$, and define \[
\Zpq_{3}:=\rG(2,T(-1))\times_{\mP(V)}\hcoY_{3},\]
with the natural projections $\pi_{G'}:\Zpq_{3}\to\rG(2,T(-1))$ and
$\pi_{3'}:\Zpq_{3}\to\hcoY_{3}$. By definition, the fiber of $\pi_{3'}$
over the points $([\bar{U}],[V_{1}])\in\hcoY_{3}\setminus(\sP_{\rho}\cup\sP_{\sigma})$
is \[
\rG(2,V/V_{1})\cap\mP(\bar{U}),\]
which are conics isomorphic to $\mP(U)\cap\rG(3,V)$ with $U=\bar{U}\wedge V_{1}$,
i.e., the fibers of $\overline{\Zpq}\to\overline{\hcoY}$ over $[U]$.
As before the fibers over $\sP_{\rho}$ and $\sP_{\sigma}$ are the
$\rho$-planes and $\sigma$-planes, respectively. 

Noting the isomorphism $\rG(2,T(-1))\simeq F(1,3,V)$, the following
lemma is clear:
\begin{lem}
There is a natural morphism $\rho_{G}:\rG(2,T(-1))\to\rG(3,V).$
\end{lem}
\vskip1cm

\begin{center}
\includegraphics[scale=0.9]{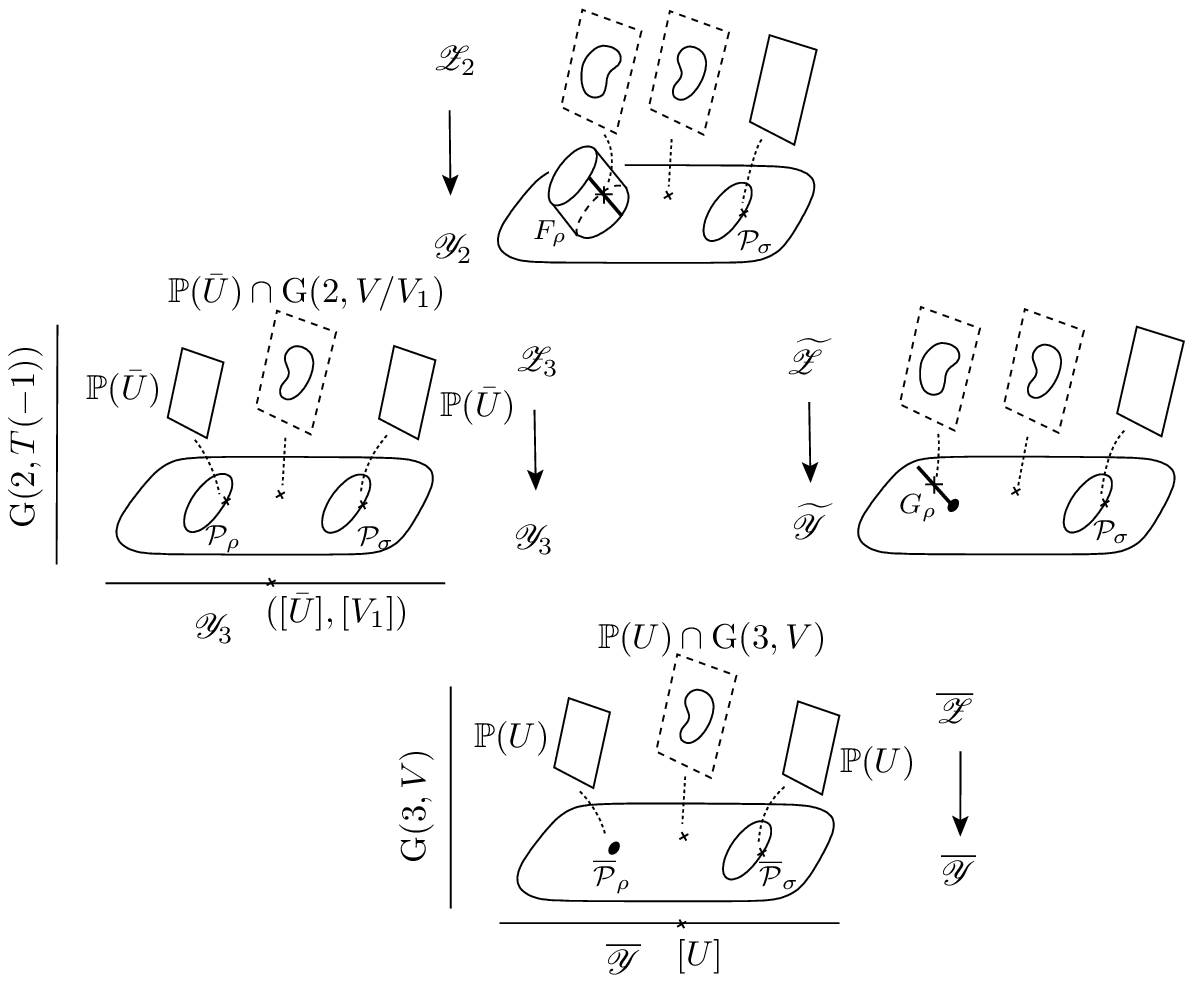}
\par\end{center}

\vskip0.5cm\begin{fcaption}\item{} \textbf{Fig.5.2. Generically conic
bundles.} Generically conic bundles in the text are schematically
described. The proper transforms of $\overline{\sP}_{\sigma}$ are
written by the same letter $\sP_{\sigma}$ for simplicity. 

\end{fcaption}

\vskip0.3cm

\subsubsection{{\itshape\bfseries Generically conic bundle $\Zpq_{2}\to\hcoY_{2}$}}

As described in Proposition \ref{pro:Resolutions-Y}, $\hcoY_{2}$
is given as the blow-up of $\hcoY_{3}$ along $\sP_{\rho}$. We denote
the exceptional divisor of the blow-up by $F_{\rho}$ (note that $F_{\rho}$
is a divisor). 
\begin{prop}
\label{pro:Y3-Y2-Frho}\begin{myitem} \item{\rm (1)} We have $\sN_{\sP_{\rho}/\hcoY_{3}}=\tS^{2}\sS^{*}\otimes\pi_{3}^{*}\sO_{\mP(V)}(1)\vert_{\sP_{\rho}}$
for the normal bundle of $\sP_{\rho}\subset\hcoY_{3}$, and hence
$F_{\rho}=\mP(\tS^{2}\sS^{*}\vert_{\sP_{\rho}}).$ 

\item{\rm (2)} The fibers of $F_{\rho}\to\sP_{\rho}$ can be identified
with the conics in the $\rho$-planes parametrized by $\sP_{\rho}$. 

\end{myitem}\end{prop}
\begin{proof}
(1) We have seen in Proposition \ref{pro:spin-OG36} that $\sP_{\rho}=\OG^{+}(3,\wedge^{2}T(-1))$,
i.e., one of the connected component of $\OG(3,\wedge^{2}T(-1))\subset\rG(3,\wedge^{2}T(-1)).$
The orthogonal Grassmannian consists maximally isotropic subspaces
with respect to the symmetric form on the universal bundle $\sS$
induced from \[
\wedge^{2}T(-1)\times\wedge^{2}T(-1)\to\wedge^{4}T(-1)\simeq\sO_{\mP(V)}(1).\]
 Hence it is given by the zero locus of the section of the bundle
$\tS^{2}\sS^{*}\otimes\pi_{3}^{*}\sO_{\mP(V)}(1)$ over $\rG(3,\wedge^{2}T(-1))$. 

\noindent (2) The points $([\bar{U}],[V_{1}])\in\sP_{\rho}$ determine
the $\rho$-planes $\mP(\bar{U})\subset\mP(\wedge^{2}(V/V_{1})).$
We can evaluate the fiber over a point $([\bar{U}],[V_{1}])\in\sP_{\rho}$
as \[
\mP(\tS^{2}\sS^{*}\vert_{([\bar{U}],[V_{1}])})=\mP(\tS^{2}\bar{U}^{*}),\]
which we identity with the conics in the $\rho$-plane.\end{proof}
\begin{prop}
Let $\rho_{2'}:\Zpq_{2}\to\Zpq_{3}$ be the blow-up of $\Zpq_{3}$
along $\pi_{3'}^{-1}(\sP_{\rho})$, and $E_{\rho}$ be its exceptional
divisor. Then $E_{\rho}\to F_{\rho}$ is the universal family of $\rho$-conics
parametrized by $F_{\rho}$. \end{prop}
\begin{proof}
This follows by considering the normal bundle of $\pi_{3'}^{-1}(\sP_{\rho})$
in $\Zpq_{3}$ carefully. We refer to \cite[Prop.4.3.4]{HoTa4} for
the proof.
\end{proof}
Now we summarize the above results into 
\begin{prop}
The natural morphism $\pi_{2'}:\Zpq_{2}\to\hcoY_{2}$ between the
blow-ups $\Zpq_{2}$ and $\hcoY_{2}$ is a generically conic bundle.
Precisely, the fibers over $\hcoY_{2}\setminus\sP_{\sigma}$ are conics
and the fibers over $\sP_{\sigma}$ are $\sigma$-planes $($where
we use the same notation $\sP_{\sigma}$ for the proper transform
of $\sP_{\sigma}$ in $\hcoY_{3})$.
\end{prop}
We may summarize generically conic bundles into the following diagram:\def\ZZZ{
\begin{xy} 
(0,0)*+{\Zpq_3}="Z3",
(15,0)*+{\Zpq_2}="Z2",
(0,-10)*+{\hcoY_3}="Y3",
(15,-10)*+{\hcoY_2}="Y2",
(0,-20)*+{\mP(V)}="P3",
(15,-20)*+{\mP(V)}="P2",
(-14.5,-10)*+{\rG(2,T(-1))}="G2",
(-30,-20)*+{\rG(3,V)}="G3",
(30,-20)*+{\widetilde{\hcoY}}="tY",
(45,-20)*+{\hcoY}="Y",
\ar_{\pi_{G'}} "Z3";"G2",
\ar_{\rho_G} "G2";"G3",
\ar_{\;\rho_{2'}} "Z2";"Z3",
\ar^{\pi_{3'}} "Z3";"Y3",
\ar^{\pi_{2'}} "Z2";"Y2",
\ar_{\rho_2} "Y2";"Y3",
\ar^{\tilde{\rho}_2} "Y2";"tY",
\ar^{\pi_3} "Y3";"P3",
\ar^{\pi_2} "Y2";"P2",
\ar_{\pi_G} "G2";"P3",
\ar^{\rho_{\widetilde{\hcoY}}} "tY";"Y",
\end{xy} }\begin{equation}
\begin{matrix}\ZZZ\end{matrix}\label{eq:diagram-ZZZ}\end{equation}
In the above diagram, we have included all the morphisms claimed in
Proposition \ref{pro:G2-Z2-to-Y}. In Fig.~5.2, we schematically
have depicted the generically conic bundles, $\overline{\Zpq}\to\overline{\hcoY}$,
$\Zpq_{3}\to\hcoY_{3}$, $\Zpq_{2}\to\hcoY_{2}$ and also $\widetilde{\Zpq}\to\widetilde{\hcoY}$
which is deduced from $\Zpq_{2}\to\hcoY_{2}$ and $\overline{\Zpq}\to\overline{\hcoY}$.

\newpage

\appendix

\section{\textup{Two Theorems on Indefinite Lattices} }

We summarize two theorems on indefinite lattices which we use in Section
\ref{sec:Fourier-Mukai-partners-K3}. 
\begin{thm}[{\cite[Theorem 1.14.2]{Nikulin}}]
\label{thm:Nikulin1} Let $L$ be an indefinite lattice and $\ell(L)$
be the minimal number of generators of $L^{\vee}/L$. If $\rk L\geq2+\ell(L)$,
then the isogeny classes of $L$ consists of $L$ itself, $\sG(L)=\left\{ L\right\} $
and the natural group homomorphism $O(L)\to O(A_{L})$ is surjective. 
\end{thm}
   
\begin{thm}[{\cite[Theorem 1.14.4]{Nikulin}}]
\label{thm:Nikulin2} Let $L$ be an even unimodular lattice with
signature $(l_{+},l_{-})$ and $M$ be an even lattice with signature
$(m_{+},m_{-})$. If {\rm (i)} $\mathrm{sgn}(L)-\mathrm{sgn}(M)>0$
$(l_{+}-m_{+}>0$, $l_{-}-m_{-}>0)$ and {\rm (ii)} $\rk L-\rk M\geq2+l(A_{M})$
hold, then primitive embedding $L\hookleftarrow M$ is unique up to
automorphism of $L$. 
\end{thm}

\section{\textup{\label{sec:AppendixB}Pl\"ucker Ideal of $\rG(3,6)$}}

Let us fix a 4-dimensional space $V_{4}$ and write the double spin
decomposition (\ref{eq:spin-decomp}) as \[
\wedge^{3}(\wedge^{2}V_{4})=\Sigma^{(3,1,1,1)}V_{4}\oplus\Sigma^{(2,2,2,0)}V_{4}\simeq\tS^{2}V_{4}\oplus\tS^{2}V_{4}^{*}.\]
We fix a basis of $V_{4}$ and write the corresponding bases of $\wedge^{4}V_{4}$
in terms of the index set $\sI=\left\{ \{i,j\}\mid1\leq i<j\leq4\right\} $
(where we regard $\{i,j\}$ as an ordered set). Then we introduce
the standard Pl\"ucker coordinate by $[p_{IJK}]\in\mP(\wedge^{3}(\wedge^{2}V_{4})).$
On the other hand, we introduce the homogeneous coordinate (which
may be called double spin coordinate) by $[v_{ij},w_{kl}]\in\mP(\tS^{2}V_{4}\oplus\tS^{2}V_{4}^{*})$
with $4\times4$ symmetric matrices $v=(v_{ij})$, $w=(w_{kl})$.
Writing the isomorphism of the above decomposition, we have a linear
relation between $[p_{IJK}]$ and $[v_{ij},w_{kl}].$ Then the Pl\"ucker
ideal $I_{G}$ of the embedding $\rG(3,\wedge^{2}V_{4})\subset\mP(\tS^{2}V_{4}\oplus\tS^{2}V_{4}^{*})$
follows from that of the standard embedding $\rG(3,\wedge^{2}V_{4})\subset\mP(\wedge^{3}(\wedge^{2}V_{4}))$. 

Let us introduce some notations. We define the signature function
$\epsilon_{IJ}\,(I,J\in\sI$) by the signature of the permutation
of the {}``ordered'' union $I\cup J$, e.g., $\{2,4\}\cup\{1,3\}=\{2,4,1,3\}$.
We also define the dual index $\check{I}\in\sI$ of $I\in\sI$ by
the property $\check{I}\cup I=\left\{ 1,2,3,4\right\} $ (here $\cup$
is the standard union). 
\begin{prop}[{\cite[Appendix A]{HoTa3}}]
 The Pl\"ucker ideal $I_{G}$ of the embedding $\rG(3,\wedge^{2}V_{4})\subset\mP(\tS^{2}V_{4}\oplus\tS^{2}V_{4}^{*})$
is generated by \begin{equation}
\begin{aligned}\vert v_{IJ}\vert-\epsilon_{I\check{I}}\epsilon_{J\check{J}}\vert w_{\check{I}\check{J}}\vert & \;\;\;\;(I,J\in\sI),\\
(v.w)_{ij},\;(v.w)_{ii}-(v.w)_{jj} & \;\;\;\;(i\not=j,1\leq i,j\leq4),\end{aligned}
\label{eq:vw-Plucker-rel}\end{equation}
where $\vert v_{IJ}\vert,\vert w_{IJ}\vert$ represent the $2\times2$
minors of $v,w$ with the rows and columns specified by $I$ and $J$.
$(v.w)_{ij}$ is the $ij$-entry of the matrix multiplication $v.w$. 
\end{prop}
\noindent For $[v,w]\in V(I_{G})\simeq\rG(3,6)$, we have 

1) $\det\, v=\det\: w$,

2) $v.w=\pm\sqrt{\det\, w}\,\id_{4}$,

3) $\rk v\not=3$ and $\rk w\not=3$,

4) $\rk v=2\Leftrightarrow\rk w=2$, and 

5) $\rk v\leq1\Leftrightarrow\rk w\leq1.$

\noindent These are easy consequences from (\ref{eq:vw-Plucker-rel}). 

$\;$

\vspace{1cm}
\noindent {\footnotesize Graduate School of Mathematical Sciences,
University of Tokyo, Meguro-ku,Tokyo 153-8914,$\,$Japan }{\footnotesize \par}

\noindent {\footnotesize e-mail addresses: hosono@ms.u-tokyo.ac.jp,
takagi@ms.u-tokyo.ac.jp} \[
\]

\end{document}